\documentclass[a4paper, 10pt, notitlepage]{article}

\usepackage{amsthm, amsmath, amssymb, latexsym}
\usepackage{mathrsfs,cite}

\theoremstyle{plain}
\newtheorem{theorem}{Theorem}[section]	
\newtheorem{corollary}[theorem]{Corollary}
\newtheorem{proposition}[theorem]{Proposition}

\theoremstyle{definition}
\newtheorem{definition}[theorem]{Definition}
\newtheorem{example}[theorem]{Example}

\theoremstyle{remark}
\newtheorem{remark}{Remark}

\DeclareMathOperator{\co}{co}
\DeclareMathOperator{\cl}{cl}
\DeclareMathOperator{\interior}{int}
\DeclareMathOperator{\dom}{dom}
\DeclareMathOperator{\core}{core}
\DeclareMathOperator{\olim}{o-lim}
\DeclareMathOperator{\supp}{supp}
\DeclareMathOperator{\Ker}{Ker}

\author{Dolgopolik M.V.\footnote{Saint Petersburg State University, Saint Petersburg, Russia}
\footnote{Institute of Problems of Mechanical Engineering, Saint Petersburg, Russia}}
\title{Abstract convex approximations of nonsmooth functions}

\begin{document}

\maketitle

\begin{abstract}
In this article we utilise abstract convexity theory in order to unify and generalize many different concepts from
nonsmooth analysis. We introduce the concepts of abstract codifferentiability, abstract quasidifferentiability and
abstract convex (concave) approximations of a nonsmooth function mapping a topological vector space to an order complete
topological vector lattice. We study basic properties of these notions, construct elaborate calculus of abstract
codifferentiable functions and discuss continuity of abstract codifferential. We demonstrate that many classical
concepts of nonsmooth analysis, such as subdifferentiability and quasidifferentiability, are particular cases of the
concepts of abstract codifferentiability and abstract quasidifferentiability. We also show that abstract convex and
abstract concave approximations are a very convenient tool for the study of nonsmooth extremum problems. We use these
approximations in order to obtain various necessary optimality conditions for nonsmooth nonconvex optimization problems
with the abstract codifferentiable or abstract quasidifferentiable objective function and constraints. Then we
demonstrate how these conditions can be transformed into simpler and more constructive conditions in some particular
cases.
\end{abstract}

\section{Introduction}

One of the first ideas in the study of the local behaviour of a function was to approximate the function under
consideration in a neighbourhood of a point by a very simple function, namely linear function, and to use this linear
function in order to study some properties of the initial function. This simple idea gave rise to the concept of
derivative, and eventually led to the development of classical differential calculus. In the twentieth century, various
generalizations of derivative were proposed in nonsmooth analysis. Most of these generalizations are just
modifications of the directional derivative or the subgradient and the subdifferential of a convex function 
(see, e.g., \cite{Clarke,RocWets,Schir,Penot,Mord}). Although these generalizations are effective tools for
solving various nonsmooth problems, they are discontinuous in the nonsmooth case. A lack of continuity and exact
calculus often makes the design of effective numerical methods very difficult.

However, there is a different way to generalize the definition of derivative. In order to study a more broad class
of functions than the class of differentiable functions, one should simply approximate a function in a neighbourhood of
a point by more broad (and inevitably more complicated) set of functions than the set of linear functions. From the
point of view of optimization, a natural candidate on the role of the set of approximating functions is the set of
convex (concave or the sum of convex and concave) functions, since the class of convex functions is the simplest and the
most profoundly studied class of functions in optimization. For a long time this simple idea had not been fulfilled in
nonsmooth analysis, until in 1988 V.F.~Demyanov introduced the concept of codifferentiable function \cite{Dem}
(that implicitly carried out this idea) in order to construct a continuous approximation of a nonsmooth function. Since
usually a continuous approximation of a nonsmooth function must be nonhomogeneous, (see the introduction in
\cite{DemRub}), we naturally come to the following definition of codifferentiable function, which is a
generalization of the concept of quasidifferentiable functon \cite{DemRub}. Let $\Omega \subset \mathbb{R}^d$ be an
open set. A function $f \colon \Omega \to \mathbb{R}$ is called codifferentiable at a point $x \in \Omega$ if there
exist convex compact sets $\underline{d} f(x)$, $\overline{d} f(x) \subset \mathbb{R}^{d+1}$ such that for any
admissible $\Delta x \in \mathbb{R}^d$ (i.~e. such that $\co\{x, x + \Delta x \} \subset \Omega$) one has
\[
f(x + \Delta x) - f(x) = \max_{(a, v) \in \underline{d} f(x)}(a + \langle v, \Delta x \rangle) +
\min_{(b, w) \in \underline{d} f(x)}(b + \langle w, \Delta x \rangle) + o(\Delta x, x),
\]
where $o(\alpha \Delta x, x) / \alpha \to 0$ as $\alpha \to +0$, and $\langle \cdot, \cdot \rangle$ is the inner
product in $\mathbb{R}^d$. In actuality, the previous definition is equivalent to the following one
(cf.~example~\ref{ExampleCodiff} below): a function $f \colon \Omega \to \mathbb{R}$ is said to be codifferentiable at a
point $x \in \Omega$ if there exist a finite convex function $\Phi \colon \mathbb{R}^d \to \mathbb{R}$ and a finite
concave function $\Psi \colon \mathbb{R}^d \to \mathbb{R}$ such that for any admissible $\Delta x \in \mathbb{R}^d$
\[
f(x + \Delta x) - f(x) = \Phi(\Delta x) + \Psi(\Delta x) + o(\Delta x, x),
\]
where $o(\alpha \Delta x, x) / \alpha \to 0$ as $\alpha \to +0$. The concept of codifferentiability appeared to be
an effective tool for solving various nonsmooth optimization problems \cite{DemRub, DemBagRub, BagUgon, BagGanUgonTor}.
Let us mention here an interesting ability of the method of codifferential descent \cite{DemRub} to ``jump over'' some
points of local minimum \cite{DemBagRub}.

The aim of this article is to take the next step and to utilize some ideas of abstract convexity in nonsmooth analysis.
Namely, we introduce and study the concepts of abstract codifferentiability, abstract quasidifferentiability, and
abstract convex and abstract concave approximations of a nonsmooth function mapping a topological vector space to an
order complete topological vector lattice. These concepts are based on the idea of an approximation of a function in a
neighbourhood of a point by an abstract convex function (or an abstract concave function, or the sum of abstract convex
and abstract concave functions). Actually, many well-known notions of nonsmooth analysis, such as
subdifferentiability, quasidifferentiability, codifferentiability, exhauster, and coexhauster, are just a particular
cases of the concepts of abstract quasidifferentiability and abstract codifferentiability. Thus, the theory presented
in the article gives us a new understanding of these notions, and allows one to present many different concepts and
results of nonsmooth analysis in a unified and convenient framework. Moreover, the theory of abstract
codifferentiability furnishes one with a useful approach to the construction and study of continuous approximations
of nonsmooth functions. Therefore, we pay a lot of attention to the problem of continuity of an abstract codifferential
and thoroughly develop the calculus of abstract codifferentiable functions.

In the article, we also derive necessary optimality conditions for various nonsmooth nonconvex optimization problems
with the use of the abstract convex and abstract concave approximations of a nonsmooth function. The author thinks that
the abstract convex and abstract concave approximations are a very effective tool for the study of various nonsmooth
constrained extremum problems, since they allow one to obtain necessary optimality conditions for these problems in a
very simple manner. As applications of the general theory, we give new characterizations of some classes of nonsmooth
functions, and obtain new necessary optimality conditions for these classes of functions.

\section{Preliminaries}

In this section we recall some basic notions from abstract convexity, and introduce several specific sets and operations
on these sets, which will simplify the exposition of the main results in the article. We assume that the reader is
familiar with some basic definitions and facts from the theory of topological vector lattices 
\cite{Schaefer, MeyerNieberg} and abstract convex analysis \cite{Rubinov, Singer, Pallashke}.

\subsection{Abstract convexity}

We recall some definitions from abstract convexity, that is used in subsequent. Let $X$ be an arbitrary nonempty set,
$E$ be a complete lattice, $f \colon X \to E$ be an arbitrary function, and let $H$ be a nonvoid set of mappings 
$h \colon X \to E$. If $h \in H$ and $h(x) \le f(x)$ for all $x \in X$, then we write $h \le f$ (or $f \ge h$).

\begin{definition}
The function $f$ is called abstract convex with respect to $H$ (or $H$-convex) if there exists a nonempty set 
$U \subset H$ such that $f(x) = \sup_{h \in U} h(x)$ for all $x \in X$.
In this case one says that the abstract convex function $f$ is generated by $U$.

The function $f$ is called abstract concave with respect to $H$ (or $H$-concave) if there exists a nonempty set 
$V \subset H$ such that $f(x) = \inf_{h \in V} h(x)$ for any $x \in X$. In the latter case the abstract concave
function $f$ is said to be generated by $V$.
\end{definition}

The set $\supp^+(f, H) = \{ h \in H \mid f \le h \}$ is called an upper support set of $f$ with respect to $H$, and 
the set $\supp^-(f, H) = \{ h \in H \mid f \ge h \}$ is referred to as a lower support set of $f$ with respect to $H$.
The set $\underline{\partial}_H f(x) = \{ h \in \supp^-(f, H) \mid h(x) = f(x)\}$ is called an
$H$-subdifferential of $f$ at $x$, and the set $\overline{\partial}_H f(x) = \{ h \in \supp^+(f, H) \mid h(x) = f(x)\}$
is referred to as an $H$-superdifferential of $f$ at $x$.

Note an obvious condition for the global minimum (maximum) of the function $f$ via abstract convex structures. Suppose
that the set $H$ contains all constant functions. Then it is easy to check that for the function $f$
to have a global minimum (maximum) value at a point $x^*$ it is necessary and sufficient that 
\begin{equation} \label{glExtrCondAbstSubDiff}
f(x^*) \in \underline{\partial}_H f(x^*) \quad (f(x^*) \in \overline{\partial}_H f(x^*)).
\end{equation}
One can suppose that only the constant function $h \equiv f(x^*)$ belongs to $H$ in order to get
(\ref{glExtrCondAbstSubDiff}).

\subsection{Special sets}

Let $X$ be an arbitrary nonvoid set and $E$ be an order complete vector lattice. We add the improper elements $+\infty$
and $-\infty$ to the vector lattice $E$, where, as usual, $+\infty$ is considered as a greatest element, and $-\infty$
is considered as a least element. Denote $\overline{E} = E \cup \{ - \infty \} \cup \{ +\infty \}$. It is clear that
$\overline{E}$ endowed with an obvious order relation is a complete lattice. Set
\begin{gather*}
x + (+\infty) = (+\infty) + x = +\infty, \quad x + (-\infty) = (-\infty) + x = -\infty, \\
\alpha (+\infty) = + \infty, \quad \alpha (-\infty) = -\infty \quad \mbox{if } \alpha > 0, \\
\alpha (+\infty) = - \infty, \quad \alpha (-\infty) = +\infty \quad \mbox{if } \alpha < 0.
\end{gather*}
We will not consider such expressions as $+\infty + (-\infty)$ or $0 (+\infty)$. For an arbitrary function 
$F \colon X \to \overline{E}$ denote  $\dom F = \{ x \in X \mid F(x) \ne - \infty, F(x) \ne + \infty \}$. 
The sum $l = p + q$ of functions $p, q \colon X \to \overline{E}$ is said to be well-defined if 
$p^{-1}(e) \cap q^{-1}(-e) = \emptyset$ when $e \in \{ +\infty, - \infty \}$. Here $p^{-1}(e)$ is the preimage of $e$
under $p$.

Let $H$ be a nonempty set of functions $h \colon X \to \overline{E}$. The set $H$ is said to be closed under addition
if for any $h_1, h_2 \in H$ the sum $h_1 + h_2$ is well-defined and belongs to $H$.

Let $\mathfrak{F}$ be a filter on $X$. Denote by $PF(X, \mathfrak{F}, \overline{E}, H)$ the set consisting of all pairs
of functions $(\Phi, \Psi)$ such that $\Phi \colon X \to \overline{E}$ is $H$-convex, $\Psi \colon X \to \overline{E}$
is $H$-concave, and there exists $S \in \mathfrak{F}$ such that $S \subset \dom \Phi \cap \dom \Psi$.

\begin{remark}
We will only consider values of the sum $\Phi + \Psi$ in a ``neighbourhood'' of a given point $x$, since the
sum $\Phi + \Psi$ will serve as an approximation of the increment of a function in this ``neighbourhood''. Therefore,
it is natural to demand that the sum $\Phi + \Psi$ is well-defined and finite only in a ``neighbourhood'' of $x$.
Thus, the filter $\mathfrak{F}$ will usually be the filter of neighbourhoods of a point $x$.
\end{remark}

In subsequent we will consider an approximation of the increment of a function by the sum of $H$-convex and
$H$-concave functions. Different pairs of $H$-convex and $H$-concave functions could define the same
approximation. Therefore it is convenient to introduce the set of equivalence classes of pairs of $H$-convex and
$H$-concave functions that define the same approximation. 

Let us introduce a binary relation $\sigma$ on the set $PF(X, \mathfrak{F}, \overline{E}, H)$. We say that 
$((\Phi_1, \Psi_1), (\Phi_2, \Psi_2)) \in \sigma$, where $(\Phi_i, \Psi_i) \in PF(X, \mathfrak{F}, \overline{E}, H)$,
$i \in \{1, 2\}$, if and only if there exists $S \in \mathfrak{F}$ such that $S \subset \dom \Phi_i \cap \dom \Psi_i$,
$i \in \{1, 2 \}$ and
\[
\Phi_1(x) + \Psi_1(x) = \Phi_2(x) + \Psi_2(x) \quad \forall x \in S.
\]
It is easy to see that $\sigma$ is an equivalence relation on $PF(X, \mathfrak{F}, \overline{E}, H)$. The quotient set
of $PF(X, \mathfrak{F}, \overline{E}, H)$ by $\sigma$ is denoted by $EPF(X, \mathfrak{F}, \overline{E}, H)$. 
If $(\Phi, \Psi) \in PF(X, \mathfrak{F}, \overline{E}, H)$, then the equivalence class of $(\Phi, \Psi)$ under $\sigma$
is denoted by $[\Phi, \Psi]$.

Since an $H$-convex (or $H$-concave) function is defined by a subset of the set $H$, one can consider the set
$PS(H, \mathfrak{F})$ instead of $PF(X, \mathfrak{F}, \overline{E}, H)$, where $PS(H, \mathfrak{F})$ is the set
consisting of all pairs $(U, V)$ of nonempty sets $U, V \subset H$ such that 
$(\sup_{h \in U} h, \inf_{p \in V} p) \in PF(X, \mathfrak{F}, \overline{E}, H)$. Let us introduce a binary relation
$\widehat{\sigma}$ on the set $PS(H, \mathfrak{F})$, which is similar to the relation $\sigma$. Define 
$((U_1, V_1), (U_2, V_2)) \in \widehat{\sigma}$, where  $(U_i, V_i) \in PS(H, \mathfrak{F})$,  $i \in \{1, 2 \}$, if and
only if 
\[
\left( (\sup_{h_1 \in U_1} h_1, \inf_{p_1 \in V_1} p_1), (\sup_{h_2 \in U_2} h_2, \inf_{p_2 \in V_2} p_2) \right) 
\in \sigma
\]
It is obvious that $\widehat{\sigma}$ is an equivalence relation on $PS(H, \mathfrak{F})$. The quotient set of 
$PS(H, \mathfrak{F})$ by $\widehat{\sigma}$ is denoted by $EPS(H, \mathfrak{F})$. If $(U, V) \in PS(H, \mathfrak{F})$,
then the equivalence class of the element $(U, V)$ under $\widehat{\sigma}$ is denoted by $[U, V]$.

Introduce the operations of addition and scalar multiplication on the set $EPF(X, \mathfrak{F}, \overline{E}, H)$. Let
$\alpha \in \mathbb{R}$ be arbitrary. Suppose that $0 \in H$ in the case $\alpha = 0$, and the set $H$ is a cone 
(i.~e. for any $h \in H$ and for all $\lambda > 0$ one has $p = \lambda h \in H$) in the case $\alpha \ne 0$. 
Denote $(-H) = \{ -h \mid h \in H \}$.

Let $(\Phi, \Psi) \in PF(X, \mathfrak{F}, \overline{E}, H)$. Define
\[
\alpha [ \Phi, \Psi ] = \begin{cases}
[\alpha \Phi, \alpha \Psi ] \in EPF(X, \mathfrak{F}, \overline{E}, H), \mbox{ if } \alpha > 0, \\
[\alpha \Psi, \alpha \Phi ] \in EPF(X, \mathfrak{F}, \overline{E}, -H), \mbox{ if } \alpha < 0, \\
[0, 0], \mbox{ if } \alpha = 0.
\end{cases}
\]
It is easy to check that the previous definition is correct in the sense that if 
$(\Phi_1, \Psi_1), (\Phi_2, \Psi_2) \in [\Phi, \Psi]$, then 
$[\alpha \Phi_1, \alpha \Psi_1] = [\alpha \Phi_2, \alpha \Psi_2]$ in the case $\alpha > 0$ and 
$[\alpha \Psi_1, \alpha \Phi_1 ] = [\alpha \Psi_2, \alpha \Phi_2 ]$ in the case $\alpha < 0$.

Suppose now that the set $H$ is closed under addition. Let 
$(\Phi_1, \Psi_1), (\Phi_2, \Psi_2) \in PF(X, \mathfrak{F}, \overline{E}, H)$. Then we set
$[\Phi_1, \Psi_1] + [\Phi_2, \Psi_2] = [\Phi_1 + \Phi_2, \Psi_1 + \Psi_2]$. It is easy to verify that the given
definition of the sum is correct.

The operations of addition and scalar multiplications on the set $EPS(H, \mathfrak{F})$ are defined in a similar way.

\begin{remark}
(i) {The construction of the sets $EPS(H, \mathfrak{F})$ and $EPF(X, \mathfrak{F}, \overline{E}, H)$ is similar to the
construction of the space of convex sets \cite{Radstrom} and the set of the differences of sublinear
functions \cite{DemRub, PalUrb}.
}

\noindent(ii) {Let $X$ be a topological vector space and $\mathfrak{F}$ be the filter of neighbourhoods of the origin.
Then we write $PF(X, \overline{E}, H)$ instead of $PF(X, \mathfrak{F}, \overline{E}, H)$ and use analogous abbreviations
for $EPF(X, \mathfrak{F}, \overline{E}, H)$, $PS(H, \mathfrak{F})$ and $EPS(H, \mathfrak{F})$.
}
\end{remark}

We need to introduce other equivalence relations on the set $PF(X, \mathfrak{F}, \overline{E}, H)$ in order to
avoid ambiguity in the definition of abstract codifferentiable function.

Let $X$ be a topological vector space (normed space) over the field of real or complex numbers, and $E$ be an order
complete Hausdorff topological vector lattice. Define a binary relation $\sigma_w$ (and $\sigma_s$) on the set 
$PF(X, \overline{E}, H)$. Let $(\Phi_i, \Psi_i) \in PF(X, \overline{E}, H)$, $i \in \{1, 2\}$ be arbitrary. Set
\[
((\Phi_1, \Psi_1), (\Phi_2, \Psi_2)) \in \sigma_w \quad 
\big( ((\Phi_1, \Psi_1), (\Phi_2, \Psi_2)) \in \sigma_s \big)
\]
if and only if $\Phi_1(0) + \Psi_1(0) = \Phi_2(0) + \Psi_2(0)$ and for any $x \in X$
\begin{gather*}
\lim_{\alpha \downarrow 0} \frac{1}{\alpha} 
(\Phi_1(\alpha x) + \Psi_1(\alpha x) - \Phi_2(\alpha x) - \Psi_2(\alpha x)) = 0 \\
\left( \lim_{x \to 0} \frac{1}{\| x \|}
(\Phi_1(x) + \Psi_1(x) - \Phi_2(x) - \Psi_2(x)) = 0 \right).
\end{gather*}
Hereafter we write $\alpha \downarrow 0$ instead of $\alpha \in \mathbb{R}$, $\alpha \to +0$. It is easy to see that
$\sigma_w$ and $\sigma_s$ are equivalence relations on the set $PF(X, \overline{E}, H)$. The quotient set of 
$PF(X, \overline{E}, H)$ by $\sigma_w$ is denoted by $EPF_w(X, \overline{E}, H)$, and the quotient set of 
$PF(X, \overline{E}, H)$ by $\sigma_s$ is denoted by $EPF_s(X, \overline{E}, H)$.
If  $(\Phi, \Psi) \in PF(X, \overline{E}, H)$, then the equivalence class of $(\Phi, \Psi)$ under $\sigma_w$ is denoted
by $[\Phi, \Psi]_w$, and the equivalence class of $(\Phi, \Psi)$ under $\sigma_s$ is denoted by $[\Phi, \Psi]_s$.

One can introduce similar equivalence relations $\widehat{\sigma}_w$ and $\widehat{\sigma}_s$ on the set $PS(H)$, and
the quotient sets $EPS_w(H)$ and $EPS_s(H)$. Also, it is easy to define the operations of addition and scalar
multiplication on the sets $EPF_w(X, \overline{E}, H)$, $EPF_s(X, \overline{E}, H)$, $EPS_w(H)$ and $EPS_s(H)$ in the
same way as we defined these operations on the sets $EPF(X, \mathfrak{F}, \overline{E}, H)$ and $EPS(H, \mathfrak{F})$.

Let us give several definitions that is useful for the study of continuity. Let, as earlier, $X$ be a nonvoid set and 
$f \colon X \to EPS(H, \mathfrak{F})$ be an arbitrary mapping (one can also consider $f \colon X \to EPS_w(H)$ or 
$f \colon X \to EPS_s(H)$). A mapping $\varphi = (\varphi_1, \varphi_2) \colon X \to PS(H, \mathfrak{F})$, where
$\varphi_i \colon X \to S(H)$, $i \in \{1, 2\}$, is said to be a selection of the mapping $f$ if 
$\varphi(x) \in f(x)$ for all $x \in X$. Here $S(H)$ is the set of all nonempty subsets of $H$.

Let $X$ and $H$ be equipped with topologies, and let $\Omega$ be a neighbourhood of a point $x \in X$.

\begin{definition}
A mapping $f \colon \Omega \to EPS(H,  \mathfrak{F})$ is called lower semicontinuous (upper semicontinuous, continuous)
at the point $x$ if there exists a selection $\varphi = (\varphi_1, \varphi_2) \colon \Omega \to PS(H,  \mathfrak{F})$
of $f$ such that the set-valued mappings $\varphi_1$, $\varphi_2$ are lower semicontinuous (upper semicontinuous,
continuous) at the point $x$. If $H$ is a metric space, then the mapping $f$ is called Hausdorff continuous at the point
$x$ if there exists a selection $\varphi = (\varphi_1, \varphi_2) \colon \Omega \to PS(H, \mathfrak{F})$ of $f$ such
that the set-valued mappings $\varphi_1$, $\varphi_2$ are Hausdorff continuous at this point.
\end{definition}

\section{Abstract codifferentiable functions}

In the following subsections we give definitions of $H$-codifferentiable and $H$-quasidifferentiable functions and
discuss related notions. Also, we show that many well-known classes of nonsmooth functions are, in fact,
$H$-codifferentiable or $H$-quasidifferentiable for particular sets $H$.

\subsection{A definition of abstract codifferentiable functions}

Hereafter, let $X$ be a Hausdorff topological vector space over the field of real or complex numbers, $E$ be an order
complete Hausdorff topological vector lattice, $H$ be a nonempty set of functions $h \colon X \to \overline{E}$, and 
$\Omega \subset X$ be an open set. Denote the closure of a subset $A \subset T$ of a topological
space $T$ by $\cl A$, and the convex hull of a subset $A \subset L$ of a linear space $L$ by $\co A$. 

\begin{definition}
A function $F \colon \Omega \to E$ is said to be weakly $H$-codifferentiable (or G\^ateaux $H$-codifferentiable, or
weakly abstract codifferentiable with respect to $H$) at a point $x \in \Omega$ if there exists an element 
$\delta F_H [x] \in EPF_w(X, \overline{E}, H)$ for which there exists a pair $(\Phi, \Psi) \in \delta F_H[x]$ such
that $\Phi(0) + \Psi(0) = 0$ and for any admissible argument increment $\Delta x \in X$ 
(i.e. $\co\{ x, x + \Delta x \} \subset ( \Omega \cap \dom \Phi \cap \dom \Psi ))$ the following holds
\[
F(x + \Delta x) - F(x) = \Phi(\Delta x) + \Psi(\Delta x) + o(\Delta x, x),
\]
where $o(\alpha \Delta x, x) / \alpha \to 0$ as $\alpha \downarrow 0$. The element $\delta F_H [x]$ is called a weak
$H$-derivative (or G\^ateaux $H$-derivative) of the function $F$ at the point $x$.
\end{definition}

It is clear that if a function $F$ is weakly $H$-codifferentiable at a point $x$, then any pair
$(\Phi, \Psi) \in \delta F_H[x]$ satisfies all assumptions of the previous definition, i.~e. the definition of 
weakly $H$-codifferentiable function does not depend on the choice of a pair $(\Phi, \Psi) \in \delta F_H(x)$.

\begin{definition}
Let $X$ be a normed space. A function $F \colon \Omega \to E$ is said to be strongly $H$-codifferentiable (or Fr\'echet
$H$-codifferentiable) at a point $x \in \Omega$ if there exists an element $F'_H [x] \in EPF_s(X, \overline{E}, H)$
for which there exists a pair $(\Phi, \Psi) \in F'_H[x]$ such that $\Phi(0) + \Psi(0) = 0$ and for any
admissible $\Delta x \in X$
\[
F(x + \Delta x) - F(x) = \Phi(\Delta x) + \Psi(\Delta x) + o(\Delta x, x),
\]
where $o(\Delta x, x) / \| \Delta x \| \to 0$ as $\Delta x \to 0$. The element $F'_H[x]$ is called a strong
$H$-derivative (or Fr\'echet $H$-derivative) of the function $F$ at the point $x$.
\end{definition}

It is easy to check that the weak (strong) $H$-derivative of a function $F \colon \Omega \to E$ at a point 
$x \in \Omega$ is uniquely defined. Also, it is clear that if a function $F \colon \Omega \to E$ is strongly
$H$-codifferentiable at a point $x$, then $F$ is weakly $H$-codifferentiable at this point and for any
$(\Phi, \Psi) \in F'_H[x]$ one has $(\Phi, \Psi) \in \delta F_H[x]$ (the opposite inclusion does not hold true in the
general case).

\begin{remark}
One can consider the definition of $H$-codifferentiation in a more general framework. Indeed, let $X$ be a vector
space, $E$ be a complete vector lattice, $\Omega \subset X$ be an arbitrary set. Denote by
\[
\core \Omega = \{ x \in \Omega \mid \forall g \in X \: \exists \alpha_g > 0 \colon x + \alpha g \in \Omega \quad
\forall \alpha \in (0, \alpha_g) \}.
\]
the algebraic interior of the set $\Omega$. Let $\mathfrak{F} = \{ S \subset X \mid 0 \in \core S \}$, and suppose that
$\core \Omega \ne \emptyset$.

A function $F \colon \Omega \to E$ is said to be order $H$-codifferentiable at a point $x \in \core \Omega$ if there
exists an element $\delta_o F_H [x] \in EPF(X, \mathfrak{F}, \overline{E}, H)$ for which there exists a pair 
$(\Phi, \Psi) \in \delta_o F_H[x]$ such that $\Phi(0) + \Psi(0) = 0$ and for any argument increment 
$\Delta x \in E$ such that  $\co\{ x, x + \Delta x \} \subset \core ( \Omega \cap \dom \Phi \cap \dom \Psi ))$ the
following holds
\[
\olim_{\alpha \downarrow 0} | F(x + \Delta x) - F(x) - \Phi(\Delta x) - \Psi(\Delta x) | / \alpha = 0,
\]
where $\olim$ stands for the order limit in the lattice $E$.

Let $X$ be a topological vector space. One can also consider the notion of $H$-codifferentiability for a function
defined on the set $\Omega \cap \mathcal{K}$, where $\mathcal{K} \subset X$ is a cone, or on a closed set $M \subset X$.
In these cases, the $H$-derivative of a function is an element of $EPF(\mathcal{K}, \mathfrak{F}, \overline{E}, H)$,
where $\mathfrak{F} = \{ S \subset \mathcal{K} \mid 0 \in \interior_{\mathcal{K}} S \}$, 
$\interior_{\mathcal{K}}$ stands for the interior of a set in the topological subspace $\mathcal{K}$ of the space
$X$, and $\mathcal{K} \subset X$ is either an arbitrary cone or some kind of a tangent cone to the set $M$.

We will not consider the generalization of $H$-codifferentiability suggested above. The interested reader can transfer
main results obtained in the article to these more general cases.
\end{remark}

Let a function $F \colon \Omega \to E$ be weakly $H$-codifferentiable at a point $x \in \Omega$, and let 
$(\Phi, \Psi) \in \delta F_H[x]$ be arbitrary. Then, by the definitions of abstract convex and abstract concave
functions, there exist nonempty sets $U, V \subset H$ such that
\begin{equation} \label{RepresOfAConvAConcaveFuncs}
\Phi(y) = \sup_{h \in U} h(y), \quad \Psi(y) = \inf_{p \in V} p(y) \quad \forall y \in X.
\end{equation}
We denote the equivalence class $[U, V]_w \in EPS_w(H)$ by $D^w_H F(x)$. The set $D^w_H F(x)$ is called a weak
$H$-codifferential (or G\^ateaux $H$-codifferential) of the function $F$ at the point $x$. It is easy to check that
$D^w_H F(x)$ does not depend on the choice of $(\Phi, \Psi) \in \delta F_H[x]$ and the choice of the sets 
$U, V \subset H$ satisfying (\ref{RepresOfAConvAConcaveFuncs}). Hence the weak $H$-codifferential of the function $F$
at the point $x$ is unique. One can analogously define a strong $H$-codifferential (or Fr\'echet $H$-codifferential)
$D_H^s F(x)$ of the function $F$ at the point $x$.

\begin{definition}
Let a function $F \colon \Omega \to E$ be weakly (strongly) $H$-codifferentiable at a point $x \in \Omega$, and suppose
that $0 \in H$. The function $F$ is said to be weakly (strongly) $H$-hypodifferentiable at $x$ if there exists an
$H$-convex function $\Phi \colon X \to \overline{E}$ such that $\delta F_H [x] = [\Phi, 0]_w$ 
($F'_H[x] = [\Phi, 0]_s$). The function $F$ is said to be weakly (strongly) $H$-hyperdifferentiable at $x$ if there
exists an $H$-concave function $\Psi$ such that $\delta F_H [x] = [0, \Psi]_w$ ($F'_H [x] = [0, \Psi]_s$). 
\end{definition}

Although the $H$-derivative of a function is unique, in the general case there exist $(\Phi_i, \Psi_i) \in F'_H[x]$, 
$i \in \{1, 2\}$ such that $[\Phi_1, \Psi_1] \ne [\Phi_2, \Psi_2]$. The following example shows the difference between
equivalence relations $\sigma$ and $\sigma_s$.

\begin{example}
Let $X = E = \mathbb{R}$, $H$ be the set of all affine functions, i.e. 
\[
H = \{ h \colon \mathbb{R} \to \mathbb{R} \mid h(x) = ax + b, \mbox{ where } a, b, x \in \mathbb{R} \},
\]
and  $F(x) = x^4$ for all $x \in \mathbb{R}$. It is clear that $F$ is strongly $H$-codifferentiable at the
point $x = 0$, and $F'_H[0] = [0, 0]_s$. Define $\Phi(x) = x^2$, $x \in \mathbb{R}$. It is easy to verify that $\Phi$
is $H$-convex and $[\Phi, 0] \ne [0, 0]$, despite the fact that $(\Phi, 0) \in F'_H[0]$, i.~e. 
$[\Phi, 0]_s = [0, 0]_s$.
\end{example}

Let us introduce the important concept of continuously $H$-codifferentiable functions. Let $H$ be endowed with
a topology.

\begin{definition}
A function $F \colon \Omega \to E$ is said to be continuously (upper semicontinuously, lower semicontinuously or, in
the case when $H$ is equipped with a metric, Hausdorff continuously) weakly $H$-codifferentiable at a point 
$x \in \Omega$ if the function $F$ is weakly $H$-codifferentiable in a neighbourhood $\mathcal{O}$ of $x$, and 
the mapping $y \to D^w_H F(y)$, $y \in \mathcal{O}$ is continuous (upper semicontinuous, lower semicontinuous, Hausdorff
continuous) at $x$. Continuously strongly $H$-codifferentiable functions are defined in the same way.
\end{definition}

\begin{definition}
Let $0 \in H$. A function $F \colon \Omega \to E$ is said to be continuously weakly $H$-hypodifferentiable at a point
$x \in \Omega$ if the function $F$ is weakly $H$-hypodifferentiable in a neighbourhood $\mathcal{O}$ of $x$ and there
exists a continuous mapping $\varphi \colon \mathcal{O} \to S(H)$ such that $(\varphi(y), 0) \in D^w_H F(y)$ for all 
$y \in \mathcal{O}$. Other types of continuity (semicontinuity) of $H$-hypodifferentiable and $H$-hyperdifferentiable
functions are defined in a similar way.
\end{definition}

\begin{remark}
It is to be mentioned that the theory of continuously $H$-codifferentiable functions is closely related to the theory
of continuous approximations of nonsmooth functions \cite{RubinovZaff, Zaffaroni}.
\end{remark}

Let us give an auxiliary definition that will be useful in subsequent.

\begin{definition}
Suppose that $X$ is a normed space, and $E$ is an order complete normed lattice. Let a function $F \colon \Omega \to E$
be weakly (strongly) $H$-codifferentiable at a point $x \in \Omega$. The weak (strong) $H$-derivative of $F$ at $x$ is
said to be Lipschitz continuous in a neighbourhood of zero (or to satisfy the Lipschitz condition in a neighbourhood of
zero) if there exists $(\Phi, \Psi) \in \delta F_H[x]$ ($(\Phi, \Psi) \in F'_H[x]$) such that the functions
$\Phi(\cdot)$ and $\Psi(\cdot)$ are Lipschitz continuous in a neighbourhood of zero.
\end{definition}

Note an obvious property of an $H$-codifferentiable function which $H$-derivative is Lipschitz continuous at the
origin.

\begin{proposition} \label{LipschitzContinuity}
Let $X$ be a normed space, $E$ be an order complete normed lattice, and a function $F \colon \Omega \to E$ be weakly
$H$-codifferentiable at a point $x \in \Omega$. Suppose that $\delta F_H[x]$ is Lipschitz continuous in a neighbourhood
of zero. Then there exists $L > 0$ such that for any admissible argument increment $\Delta x \in X$ there exists
$\alpha_0 > 0$ such that
\[
\| F(x + \alpha \Delta x)  - F(x) \| \le L \alpha \| \Delta x \| \quad \forall \alpha \in (0, \alpha_0).
\]
Moreover, if $F$ is strongly $H$-codifferentiable at $x$ and $F'_H[x]$ is Lipschitz continuous in a neighbourhood of
zero, then there exists $L > 0$ and $r > 0$ such that
\[
\| F(x + \Delta x) - F(x) \| \le L \| \Delta x \| \quad \forall \Delta x \in X, \| \Delta x \| \le r,
\]
and, in particular, the function $F$ is continuous and calm at the point $x$.
\end{proposition}

\subsection{Examples of abstract codifferentiable functions}

In this subsection we show that some well-known classes of nonsmooth functions are $H$-codifferentiable for particular
sets $H$.

\begin{example}
Let $X$ be a normed space, $E$ be an order complete normed lattice, and let $\mathcal{B}(X, E) \subset H$, i.~e. 
$H$ includes the space of all bounded linear operators mapping $X$ to $E$. Then it is clear that if a function 
$F \colon \Omega \to E$ is G\^ateaux (Fr\'echet) differentiable at a point $x \in \Omega$, then $F$ is weakly
(strongly) $H$-codifferentiable at this point. Moreover, if $\delta F[x]$ ($F'[x]$) is the G\^ateaux (Fr\'echet)
gradient of the function $F$ at the point $x$, then
\begin{gather*}
\delta F_H[x] = [\delta F[x], 0]_w = [0, \delta F[x]]_w,
\quad D^w_H F(x) = [ \{ \delta  F[x] \}, \{ 0 \} ]_w = [ \{ 0 \}, \{ \delta  F[x] \} ]_w \\
(F'_H[x] = [F'[x], 0]_s = [0, F'[x]]_s, \quad D^s_H F(x) = [ \{ F'[x] \}, \{ 0 \} ]_s = [ \{ 0 \}, \{ F'[x] \} ]_s).
\end{gather*}
The space $H$ can be equipped with the standard operator norm. Then it is easy to see that if the function 
$F \colon \Omega \to E$ is continuously G\^ateaux (Fr\'echet) differentiable at a point $x \in \Omega$, then $F$ is
Hausdorff continuously weakly (strongly) $H$-codifferentiable at this point.
\end{example}

\begin{example} \label{ExampleCodiff}
Let $X$ be a real normed space, $E = \mathbb{R}$, and let $H$ be the set of all continuous affine functions
mapping $X$ to $\mathbb{R}$, i.~e.
\[
H = \{ h \colon X \to \mathbb{R} \mid h(\cdot) = a + p(\cdot), a \in \mathbb{R}, p \in X^* \},
\]
where, as usual, $X^*$ is the topological dual space of $X$. The set $H$ can be identified with the space 
$\mathbb{R} \times X^*$. Thus, $H$ is a linear space that can be endowed with the norm
\[
\| h \|_r = \left( |a|^r + \| p \|^r \right)^{\frac 1r}, \quad h = (a, p) \in H = \mathbb{R} \times X^*,
\]
where $1 \le r < \infty$, or $\| h \|_{\infty} = \max\{|a|, \| p \| \}$.

It is well-known (see~\cite{EkelandTemam}, proposition I.3.1) that a function $\Phi \colon X \to \overline{\mathbb{R}}$
is abstract convex (abstract concave) with respect to the set $H$ under consideration if and only if $\Phi$ is a proper
lower semicontinuos convex function (proper upper semicontinuous concave function). 
Hence, a function $F \colon \Omega \to \mathbb{R}$ is weakly $H$-codifferentiable at a point $x \in \Omega$ if and only
if there exist a proper lower semicontinuos (l.s.c.) convex function $\Phi \colon X \to \overline{\mathbb{R}}$ and 
a proper upper semicontinuous (u.s.c.) concave function $\Psi \colon X \to \overline{\mathbb{R}}$ such that 
$0 \in \interior (\dom \Phi \cap \dom \Psi)$, $\Phi(0) + \Psi(0) = 0$, and for any admissible argument 
increment $\Delta x \in X$
\[
F(x + \Delta x) - F(x) = \Phi(\Delta x) + \Psi(\Delta x) + o(\Delta x, x),
\]
where $o(\alpha \Delta x, x) / \alpha \to 0$ as $\alpha \downarrow 0$.

We need the following proposition in order to give another characterization of $H$-codifferentiability for the set $H$
under consideration. Let $x \in X$ and $r > 0$. Denote $\mathcal{O}(x, r) = \{ y \in X \mid \| x - y \| < r \}$ and 
$B(x, r) = \{ y \in X \mid \| x - y \| \le r \}$.

\begin{proposition} 
Let $X$ be a real Banach space and $f \colon X \to \overline{\mathbb{R}}$ be a proper l.s.c. convex function such that 
$0 \in \interior \dom f$. Then there exist $r > 0$ and a convex bounded set $A \subset \mathbb{R} \times X^*$
that is compact in the topological product $(\mathbb{R}, \tau) \times (X^*, w^*)$ and such that
\begin{equation} \label{representOfConvFunc}
f(x) = \max_{(a, p) \in A} (a + p(x)) \quad \forall x \in B(x, r).
\end{equation}
Here $\tau$ is the standard topology on $\mathbb{R}$ and $w^*$ is the weak${}^*$ topology on $X^*$.
\end{proposition}

\begin{proof}
From the facts that the space $X$ is complete, $0 \in \interior \dom f$ and $f$ is a proper l.s.c. convex function it
follows that $f$ is continuous on $\interior \dom f$ (\cite{EkelandTemam}, corollary I.2.5), and for any 
$x \in \interior \dom f$ one has $\underline{\partial} f(x) \ne \emptyset$ (\cite{EkelandTemam}, proposition I.5.2),
where $\underline{\partial} f(x)$ is the subdifferential of the convex function $f$ at a point $x$. Thus, there exist 
$r > 0$ and $C > 0$ such that 
\begin{equation} \label{boundednessOfConvFunc}
|f(x)| \le C \quad \forall x \in \mathcal{O}(0, 4r).
\end{equation}

With the use of the definition of the subgradient of a convex function it is easy to show that there exists $M > 0$ 
($M \le C / r$) such that for all $x \in \mathcal{O}(0, 2r)$
\begin{equation} \label{boundednessOfSubdiff}
\| p \| \le M \quad \forall p \in \underline{\partial} f(x),
\end{equation}
i.~e. the subdifferential of $f$ is bounded on $\mathcal{O}(0, 2r)$.

Let a mapping $B(0, r) \ni x \to p[x] \in X^*$ be such that $p[x] \in \underline{\partial} f(x)$. Note that such
mapping exists, since $\underline{\partial} f(x) \ne \emptyset$ for all $x \in B(0, r)$. Introduce the set
\[
A = \cl \co \{ (a, p) \in \mathbb{R} \times X^* \mid a = f(x) - p[x](x), p = p[x], x \in B(0, r) \}.
\]
Here the closure is taken in the topology $\tau \times w^*$. The set $A$ is obviously convex. Taking into account
(\ref{boundednessOfConvFunc}) and (\ref{boundednessOfSubdiff}) one has that
\begin{equation} \label{inclusionInCompSet}
A \subset [ -C - r M, C + r M ] \times \{ p \in X^* \mid \| p \| \le M \}.
\end{equation}
Therefore the set $A$ is bounded and compact in the topology $\tau \times w^*$, since the set 
$\{ p \in X^* \mid \| p \| \le M \}$ is weak${}^*$ compact by the Banach-Alaoglu theorem, and the set on the
right-hand side of (\ref{inclusionInCompSet}) is compact in the topology $\tau \times w^*$ as the direct product of two
compact sets.

By the definition of the subgradient of a convex function one has that
$$
  f(y) \ge f(x) - p[x](x) + p[x](y) \quad \forall y \in X, \; \forall x \in B(0, r)
$$
and the last inequality turns into an equality when $y = x$. Hence the validity of (\ref{representOfConvFunc}) follows
from the definition of the set $A$.
\end{proof}

\begin{corollary} \label{represOfFamOfConvFunc}
Let $X$ be a Banach space and $\{ f_{\lambda} \}$, $\lambda \in \Lambda$ be a family of proper l.s.c.
convex functions mapping $X$ to $\overline{\mathbb{R}}$. Suppose that there exist $\rho > 0$ and $C_{\lambda} > 0$,
$\lambda \in \Lambda$ such that $|f_{\lambda}(x)| \le C_{\lambda}$ for all $x \in \mathcal{O}(0, \rho)$ and 
$\lambda \in \Lambda$. Then there exist $r > 0$ (depending only on $\rho$) and a family $\{ A_{\lambda} \}$, 
$\lambda \in \Lambda$ of subsets of the space $\mathbb{R} \times X^*$ such that for any $\lambda \in \Lambda$ the set
$A_{\lambda}$ is nonempty, convex, bounded and compact in the topology $\tau \times w^*$, and the following holds
\[
f_{\lambda}(x) = \max_{(a, p) \in A_{\lambda}}(a + p(x)) \quad \forall x \in B(0, r).
\]
\end{corollary}

Let us give a description of $H$-codifferentiable functions for the set $H$ under considerations. Suppose that the
normed space $X$ is complete. By virtue of the previous proposition one has that a function 
$F \colon \Omega \to \mathbb{R}$ is weakly $H$-codifferentiable at a point $x \in \Omega$ if and only if there exist
bounded convex sets $A, B \subset \mathbb{R} \times X^*$ that are compact in the topology $\tau \times w^*$ and such
that for any admissible argument increment $\Delta x \in X$
\[
F(x + \Delta x) - F(x) = \max_{(a, p) \in A} (a + p(\Delta x)) + 
\min_{(b, q) \in B} (b + q(\Delta x)) + o(\Delta x, x),
\]
where $o(\alpha \Delta x, x) / \alpha \to 0$ as $\alpha \downarrow 0$. Thus, the function $F$ is weakly
$H$-codifferentiable at a point $x \in \Omega$ if and only if it is codifferentiable at this point 
(see~\cite{Dolgopolik1, DemRub, Kuntz, Zaffaroni}). Also it is easy to show that the function $F$ is Hausdorff
continuously weakly $H$-codifferentiable at a point $x \in \Omega$ if and only if $F$ is continuously codifferentiable
at this point. Moreover, $F$ is strongly $H$-codifferentiable if and only if $F$ is codifferentiable uniformly in
directions (see~\cite{DemRub, Dolgopolik1}). If $F$ is strongly $H$-codifferentiable at $x$, then we will call it
Fr\'echet (or strongly) codifferentiable at $x$.
\end{example}

\begin{remark}
The concept of codifferentiability in Banach lattices \cite{Zaffaroni} is, in fact, the particular case of
$H$-codifferentiability, when the set $H$ consists of all affine functions $h \colon X \to E$, 
$h(x) = a + A x$, where $a \in E$ and $A \colon X \to E$ is a linear operator.
\end{remark}

\begin{example} \label{ExampleCoexhauster}
Let $X$ be a real Banach space, $E = \mathbb{R}$, and let the set $H$ consist of all proper l.s.c. convex
functions $h \colon X \to \overline{\mathbb{R}}$ such that $0 \in \interior \dom h$. In this example we only consider
$H$-hyperdifferentiable functions, since the set of all $H$-hyperdifferentiable functions contains a certain class of
nonsmooth functions.

Suppose that a function $F \colon \Omega \to \mathbb{R}$ is weakly $H$-hyperdifferentiable at a point $x \in \Omega$,
i.~e. there exists a set $U \subset H$ such that for any admissible argument increment $\Delta x \in X$
\[
F(x + \Delta x) - F(x) = \inf_{h \in U} h(\Delta x) + o(\Delta x, x),
\]
where $o(\alpha \Delta x, x) / \alpha \to 0$ as $\alpha \downarrow 0$. Suppose also that there exist $\rho > 0$ and 
$C_h > 0$, $h \in U$ such that
\begin{equation} \label{AddPropForCoex}
|h(x)| \le C_h \quad \forall x \in \mathcal{O}(0, \rho), \forall h \in U.
\end{equation}
Then, applying corollary \ref{represOfFamOfConvFunc} one gets that there exists a family of convex bounded sets 
$A_h \subset \mathbb{R} \times X^*$, $h \in U$, which are compact in the topology $\tau \times w^*$ and such that for
any admissible argument increment $\Delta x \in X$
\[
F(x + \Delta x) - F(x) = \inf_{h \in U} \max_{(a, p) \in A_h} (a + p(\Delta x)) + o(\Delta x, x),
\]
where $o(\alpha \Delta x, x) / \alpha \to 0$ as $\alpha \downarrow 0$. Thus, the family 
$\overline{E}(x) = \{ A_h \subset \mathbb{R} \times X^* \mid h \in U \}$, that is said to be generated by $U$,
is a Dini upper coexhauster of the function $F$ at the point $x$ \cite{Demyanov}. Therefore, as it is easy to
check, a function $F$ has a Dini upper coexhauster at a point $x$ if and only if $F$ is weakly $H$-hyperdifferentiable
at this point and there exist $(U, \{ 0 \}) \in D^w_H F[x]$, $\rho > 0$ and $C_h > 0$, $h \in U$ such that
(\ref{AddPropForCoex}) holds true. The notion of coexhauster of a nonsmooth function was introduced by Aban'kin in
\cite{Abankin}, where the functions having upper coexhauster were called $H$-hyperdifferentiable
(see~also~\cite{Demyanov}).

We will say that a family of nonempty convex, bounded and compact in the topology $\tau \times w^*$ subsets
$\overline{E}(x)$ of $\mathbb{R} \times X^*$ is a Fr\'echet upper coexhauster of $F$ at $x$ if $F$ is Fr\'echet
$H$-hyperdifferentiable at this point and there exists $(U, \{0\}) \in D^s_H F(x)$ such that $\overline{E}(x)$ is
generated by $U$.
\end{example}

\begin{remark}
One can also consider an example, that is similar to the previous one, where the set $H$ coincides with the set of all
proper u.s.c. concave functions $h \colon X \to \overline{\mathbb{R}}$ such that $0 \in \interior \dom h$.
In this case, if a function $F \colon \Omega \to \mathbb{R}$ has a Dini lower coexhauster at a point $x \in \Omega$ then
$F$ is weakly $H$-hypodifferentiable at this point.
\end{remark}

\subsection{Abstract quasidifferentiable functions}

It is easy to verify that the following proposition about the directional derivative of an $H$-codifferentiable
function holds true.

\begin{proposition}
Let $X$ be a topological vector space (normed space), a function $F \colon \Omega \to E$ be weakly (strongly)
$H$-codifferentiable at a point $x \in \Omega$. Suppose also that there exists $(\Phi, \Psi) \in \delta F_H[x]$
($(\Phi, \Psi) \in F'_H[x]$) such that the functions $\Phi$ and $\Psi$ are Dini (Hadamard)
directionally differentiable at the origin. Then the function $F$ is Dini (Hadamard) directionally differentiable at
the point $x$ and 
\[
F'(x, g) = \Phi'(0, g) + \Psi'(0, g) \quad \forall g \in X.
\]
Here $F'(x, \cdot)$, $\Phi'(0, \cdot)$ and $\Psi'(0, \cdot)$ are the Dini (Hadamard) directional derivatives of the
functions $F$, $\Phi$ and $\Psi$, respectively.
\end{proposition}

\begin{corollary} \label{DiniDDofHcodiffFunc}
Let $X$ be a topological vector space, a function $F \colon \Omega \to E$ be weakly $H$-codifferentiable at a point 
$x \in \Omega$. Suppose that any function $h \in H$ is positively homogeneous of degree one (p.h.). Then the function
$F$ is Dini directionally differentiable at the point $x$ and for any $(\Phi, \Psi) \in \delta F_H[x]$ one has
\[
F'(x, g) = \Phi(g) + \Psi(g) \quad \forall g \in X.
\]
\end{corollary}

\begin{remark}
For more details on Dini and Hadamard directional derivatives see, e.g., \cite{Demyanov, DemRub}.
\end{remark}

The previous corollary motivates us to introduce the definition of $H$-quasidifferentiable (or abstract
quasidifferentiable with respect to $H$) function. Suppose that any function $h \in H$ is p.h. (then any $H$-convex or
$H$-concave function is also p.h., and the equivalence relations $\sigma$, $\sigma_w$ and $\sigma_s$ coincide).

\begin{definition}
A function $F \colon X \to E$ is said to be Dini (Hadamard) $H$-qusidifferentiable at a point $x \in \Omega$ if $F$ is
Dini (Hadamard) directionally differentiable at this point and there exists an element 
$\mathcal{D}_H F(x) \in EPF(X, \overline{E}, H)$ such that for any $(p, q) \in \mathcal{D}_H F(x)$
\[
F'(x, g) = p(g) + q(g) \quad \forall g \in X,
\]
where $F'(x, \cdot)$ is the Dini (Hadamard) directional derivative of $F$ at $x$.
\end{definition}

The element $\mathcal{D}_H F(x)$ from the definition of Dini (Hadamard) $H$-quasidifferentiable function is called a
Dini (Hadamard) $H$-quasidifferential of the function $F$ at the point $x$. It is clear that $\mathcal{D}_H F(x)$ is
uniquely defined.

\begin{definition}
Let a function $F \colon \Omega \to E$ be Dini (Hadamard) $H$-quasidifferentiable at a point $x \in \Omega$, and
suppose that $0 \in H$. The function $F$ is said to be Dini (Hadamard) $H$-subdifferentiable at the point $x$ if there
exists an $H$-convex function $p \colon X \to E$ such that $\mathcal{D}_H F(x) = [p, 0]$. The function $F$ is
said to be Dini (Hadamard) $H$-superdifferentiable at the point $x$ if there exists an $H$-concave function 
$q \colon X \to E$ such that $\mathcal{D}_H F(x) = [0, q]$.
\end{definition}

Note a connection between $H$-quasidifferentiable functions and $H$-codifferentiable functions. It is clear that 
a function $F \colon \Omega \to E$ is weakly $H$-codifferentiable at a point $x \in \Omega$ if and only if $F$ is Dini
$H$-quasidifferentiable at this point. Also, it is easy to see that if $F$ is Dini $H$-quasidifferentiable at a point
$x \in \Omega$, and there exists $(p, q) \in D_H F(x)$ such that $p$ and $q$ are Lipschitz continuous in a neighbourhood
of zero, then $F$ is Hadamard $H$-quasidifferentiable at $x$. The following proposition, which is, partly, a
generalization of theorem 2.1 from \cite{PalRecUrb}, reveals a connection between strongly $H$-codifferentiable
functions and Hadamard $H$-quasidifferentiable functions.

\begin{proposition}
Let $X$ be a normed space, $E$ be an order complete normed lattice, and $F \colon \Omega \to E$ be an arbitrary
function. For the function $F$ to be Hadamard $H$-quasidifferentiable at a point $x \in \Omega$ it is sufficient and,
in the case when $X$ is finite dimensional, necessary that $F$ is strongly $H$-codifferentiable at this point and for
any $(\Phi, \Psi) \in F'_H[x]$ the sum $\Phi + \Psi$ is finite and continuous on $X$.
\end{proposition}

\begin{proof}
Sufficiency. Let $(\Phi, \Psi) \in F'_H[x]$, $g \in X$ and sequences $\{ g_n \} \subset X$, 
$\{ \alpha_n \} \subset (0, + \infty)$ such that $g_n \to g$ and $\alpha_n \to 0$ as $n \to \infty$ be arbitrary.
From the facts that $(\Phi, \Psi) \in F'_H[x]$ and the sum $\Phi + \Psi$ is continuous it follows that
\[
\frac{1}{\alpha_n \| g_n \|} \| F(x + \alpha_n g_n) - F(x) - \Phi(\alpha_n g_n) - \Psi(\alpha_n g_n) \| \to 0
\]
and $\| \Phi( g_n ) + \Psi(g_n) - \Phi(g) - \Psi(g) \| \to 0$ as $n \to \infty$. Consequently
\begin{multline*}
\left\| \frac{F(x + \alpha_n g_n) - F(x)}{\alpha_n} - \Phi(g) - \Psi(g) \right\| \le
\| \Phi(g_n) + \Psi(g_n) - \Phi(g) - \Psi(g) \| + \\
+ \| g_n \| \frac{1}{\alpha_n \| g_n \|} \| F(x + \alpha_n g_n) - F(x) - \Phi(\alpha_n g_n) - \Psi(\alpha_n g_n) \|
\to 0
\end{multline*}
as $n \to \infty$. Therefore the function $F$ is Hadamard $H$-quasidifferentiable at the point $x$ and 
$(p, q) \in \mathcal{D}_H F(x)$ if and only if $(p, q) \in F'_H[x]$.

Necessity. Ab absurdo, suppose that $F$ is not strongly $H$-codifferentiable at the point $x$. Fix an arbitrary 
$(p, q) \in \mathcal{D}_H F(x)$. It is clear that there exist $\varepsilon > 0$ and a sequence of admissible argument
increments $\{ \Delta x_n \} \subset X$ such that $\| \Delta x_n \| \to 0$ and for any $n \in \mathbb{N}$
\begin{equation} \label{notHCodiff}
\frac{1}{\| \Delta x_n \|}
\left\| F(x + \Delta x_n) - F(x) - p(\Delta x_n) - q(\Delta x_n) \right\| > \varepsilon.
\end{equation}
Denote $\alpha_n = \| \Delta x_n \|$, $g_n = \Delta x_n / \alpha_n$. Applying the fact that $X$ is finite dimensional
one gets that there exists a subsequence $\{ g_{n_k} \}$ converging to some $g^* \in X$, $\| g^* \| = 1$.

From the fact that $F$ is Hadamard $H$-quasidifferentiable it follows that there exists $k_1 \in \mathbb{N}$ such for
all $k > k_1$ one has
\[
\left\| \frac{F(x + \alpha_{n_k} g_{n_k}) - F(x)}{\alpha_{n_k}} - p(g^*) - q(g^*) \right\| < \frac{\varepsilon}{4}.
\]
It is well-known and easy to check, that the directional derivative $F'(x, g)$ of the Hadamard directionally
differentiable function $F$ is continuous with respect to $g$. Therefore the sum $p + q$ is continuous on $X$. Hence,
there exists $k_2 \in \mathbb{N}$ such that for any $k > k_2$ one has 
$\| p(g_{n_k}) + q(g_{n_k}) - p(g^*) - q(g^*) \| < \varepsilon / 4$. Taking into account the fact that 
$(p, q) \in \mathcal{D}_H F(x)$ one gets that for any $k > \max\{ k_1, k_2 \}$
\begin{multline*}
\frac{1}{\| \Delta x_{n_k} \|}
\left\| F(x + \Delta x_{n_k}) - F(x) - p(\Delta x_{n_k}) - q(\Delta x_{n_k}) \right\| \le \\
\le \left\| \frac{F(x + \alpha_{n_k} g_{n_k}) - F(x)}{\alpha_{n_k}} - p(g^*) - q(g^*) \right\| + \\
+ \left\| p(g_{n_k}) + q(g_{n_k}) - p(g^*) - q(g^*) \right\| \le \frac{\varepsilon}{4}
+ \frac{\varepsilon}{4} = \frac{\varepsilon}{2},
\end{multline*}
which contradicts (\ref{notHCodiff}). Thus, the function $F$ is strongly $H$-codifferentiable at the point $x$ and,
taking into account the fact that the equivalence relations $\sigma$ and $\sigma_s$ coincide in the case when any 
$h \in H$ is p.h., one gets that $(p, q) \in \mathcal{D}_H F(x)$ if and only if $(p, q) \in F'_H[x]$. Furthermore, for
any $(\Phi, \Psi) \in F'_H[x]$ the sum $\Phi + \Psi$ is finite and continuous on $X$, since 
for any $(p, q) \in \mathcal{D}_H F(x)$ the sum $p(\cdot) + q(\cdot) = F'(x, \cdot)$ is finite and continuous on $X$.
\end{proof}

\begin{remark}
It is to be mentioned that the notion of strong $H$-codifferentiability in the case when any function $h \in H$ is
positively homogeneous of degree one is closely related to the notion of semidifferentiability introduced in
\cite{Gianessi}. 
\end{remark}

Let us briefly discuss two well-known examples of $H$-quasidifferentiable functions. Let $X$ be a locally convex
Hausdorff topological vector space over the real field and $E = \mathbb{R}$. It is easy to verify that if $H = X^*$,
then a function $F \colon \Omega \to \mathbb{R}$ is Dini $H$-quasidifferentiable at a point $x \in \Omega$ if and only
if $F$ is quasidifferentiable at this point (see~\cite{DemRub, PalRecUrb, Uderzo1}). 

Suppose now that $H$ consists of all finite l.s.c. positively homogeneous convex functions $h \colon X \to \mathbb{R}$
(or u.s.c. positively homogeneous concave functions $h \colon X \to \mathbb{R}$). Then one can show that a function 
$F \colon \Omega \to \mathbb{R}$ is Dini $H$-superdifferentiable ($H$-subdifferentiable) at a point $x \in \Omega$ if
and only if there exists an upper exhauster (lower exhauster) \cite{Demyanov} of a functions $F$ at this
point. 

\begin{remark}
Note that the notion of quasidifferentiable functions in order complete vector lattices \cite{DemRub} coincide
with the notion of $H$-quasidifferentiable functions for the set $H = \mathcal{B}(X, Y)$. Also, the notion of
quasidifferentiable in the generalized sense functions introduced in \cite{Ishizuka} is the particular case of the
notion of $H$-quasidifferentiable functions, when the set $H$ consists of all finite l.s.c. positively homogeneous
convex and finite u.s.c. positively homogeneous concave functions.
\end{remark}

\subsection{Abstract convex approximations of nonsmooth functions}

In this section we consider the concept of abstract convex approximations of nonsmooth functions, that is closely
related to the notion of $H$-codifferentiability. These approximations are a very convenient tool for studying various
kinds of optimization problems. We will use them to derive necessary conditions for an extremum of 
an $H$-codifferentiable function. The notion of abstract convex approximation is a natural generalization of the
notion of convex approximation (see~\cite{Dolgopolik2} and references therein).

Let, as earlier, $H$ be a nonempty set of functions mapping $X$ to $\overline{E}$, and let $F \colon \Omega \to E$ be an
arbitrary function.

\begin{definition}
An $H$-convex function $\varphi \colon X \to \overline{E}$ is called a weak upper $H$-convex approximation (or
weak upper abstract convex approximation with respect to $H$) of the function $F$ at a point $x \in \Omega$ if
\begin{enumerate}
\item{$\varphi(0) \ge 0$ and $0 \in \interior \dom \varphi$;}

\item{for any $\Delta x \in X$ there exist $\alpha_0 > 0$ and a function $\beta \colon (0, \alpha_0) \to E$ such that 
$\co\{x, x + \alpha_0 \Delta x\} \subset \Omega \cap \dom \varphi$, $\beta(\alpha) \to 0$ as 
$\alpha \downarrow 0$ and
\[
F(x + \alpha \Delta x) - F(x) \le \varphi(\alpha \Delta x) + \alpha \beta(\alpha) \quad
\forall \alpha \in [0, \alpha_0).
\]
}
\end{enumerate}
\end{definition}

\begin{definition}
An $H$-concave function $\psi \colon X \to \overline{E}$ is referred to as a weak lower $H$-concave approximation (or
weak lower abstract concave approximation with respect to $H$) of the function $F$ at a point $x \in \Omega$ if
\begin{enumerate}
\item{$\psi(0) \le 0$ and $0 \in \interior \dom \psi$;}

\item{for any $\Delta x \in X$ there exist $\alpha_0 > 0$ and a function $\beta \colon (0, \alpha_0) \to E$ such that 
$\co\{x, x + \alpha_0 \Delta x\} \subset \Omega \cap \dom \psi$, $\beta(\alpha) \to 0$ as 
$\alpha \downarrow 0$ and
\[
F(x + \alpha \Delta x) - F(x) \ge \psi(\alpha \Delta x) - \alpha \beta(\alpha) \quad \forall \alpha \in [0, \alpha_0).
\]
}
\end{enumerate}
\end{definition}

\begin{definition}
Let $X$ be a normed space. An $H$-convex function $\varphi \colon X \to \overline{E}$ is called a strong upper
$H$-convex approximation of the function $F$ at $x \in \Omega$ if
\begin{enumerate}
\item{$\varphi(0) \ge 0$ and $0 \in \interior \dom \varphi$;}

\item{there exists $r > 0$ and a function $\beta \colon B(0, r) \to E$ such that $\beta(\Delta x) \to 0$ as 
$\Delta x \to 0$ and
\[
F(x + \Delta x) - F(x) \le \varphi(\Delta x) + \| \Delta x \| \beta(\Delta x) \quad \forall \Delta x \in B(0, r).
\]
}
\end{enumerate}
\end{definition}

One can also define a strong lower $H$-concave approximation of the function $F$ at a point $x \in \Omega$.

It is natural to expect that an upper $H$-convex approximation (lower $H$-concave approximation) does not usually
provide enough information about the behaviour of the function $F$ in a neighbourhood of a point $x$. Therefore we have
to use various families of upper $H$-convex (lower $H$-concave) approximations. Families of these approximations that
are of the most importance for the study of optimization problems is called an exhaustive families.

\begin{definition}
A family $\{ \varphi_{\lambda} \}$, $\lambda \in \Lambda$ of weak upper $H$-convex approximations of the function $F$
at a point $x \in \Omega$ is said to be exhaustive if $\inf_{\lambda \in \Lambda} \varphi_{\lambda}(0) = 0$ and 
for any admissible $\Delta x \in X$
\[
F(x + \Delta x) - F(x) = \inf_{\lambda \in \Lambda} \varphi_{\lambda} (\Delta x) + o(\Delta x, x),
\]
where $o(\alpha \Delta x, x) / \alpha \to 0$ as $\alpha \downarrow 0$.
\end{definition}

\begin{definition}
A family $\{ \psi_{\lambda} \}$, $\lambda \in \Lambda$ of weak lower $H$-concave approximations of the function $F$
at a point $x \in \Omega$  is referred to as exhaustive if $\sup_{\lambda \in \Lambda} \psi_{\lambda}(0) = 0$ and for
any $\Delta x \in X$
\[
F(x + \Delta x) - F(x) = \sup_{\lambda \in \Lambda} \psi_{\lambda} (\Delta x) + o(\Delta x, x),
\]
where $o(\alpha \Delta x, x) / \alpha \to 0$ as $\alpha \downarrow 0$.
\end{definition}

\begin{definition}
Let $X$ be a normed space. A family $\{ \varphi_{\lambda} \}$, $\lambda \in \Lambda$ of strong upper $H$-convex
approximations of the function $F$ at a point $x \in \Omega$  is said to be exhaustive if
$\inf_{\lambda \in \Lambda} \varphi_{\lambda}(0) = 0$ and for any $\Delta x \in X$
\[
F(x + \Delta x) - F(x) = \inf_{\lambda \in \Lambda} \varphi_{\lambda} (\Delta x) + o(\Delta x, x),
\]
where $o(\Delta x, x) / \| \Delta x \| \to 0$ as $\Delta x \to 0$. 
\end{definition}

The exhaustive family of strong lower $H$-concave approximations is defined in a similar way.

The following proposition reveals an obvious connection between upper $H$-convex (lower $H$-concave) approximations
and $H$-codifferentials.

\begin{proposition}
Let the set $H$ be closed under addition, and let for any $h \in H$ one has $0 \in \interior \dom h$. Suppose that a
function $F \colon \Omega \to E$ is weakly (strongly) $H$-codifferentiable at a point $x \in \Omega$. 
Then for any $(\Phi, \Psi) \in \delta F_H[x]$ and for all $h \in \supp^- (\Phi, H)$ and  $p \in \supp^+ (\Psi, H)$ the
function $\Phi + p$ is a weak (strong) upper $H$-convex approximation of $F$ at $x$ and the function  $h + \Psi$ is a
weak (strong) lower $H$-concave approximation of $F$ at $x$. Moreover, for any $(\Phi, \Psi) \in \delta F_H[x]$ and for
any $U, V \subset H$ such that $\Phi$ is generated by $U$ and $\Psi$ is generated by $V$ the family $\{ \Phi + p \}$, 
$p \in V$ is an exhaustive family of weak (strong) upper $H$-convex approximations of $F$ at $x$, and the family 
$\{ h + \Psi \}$, $h \in U$ is an exhaustive family of weak (strong) lower $H$-concave approximations of $F$ at $x$.
\end{proposition}

\section{Calculus of abstract codifferentiable functions}\label{CalcOfHCodiff}

In this section we discuss the problem of computing $H$-codifferentials and construct the $H$-codifferential
calculus. We also consider the problem of continuity of $H$-codifferentials, which is very important for practical
applications. We study only the Hausdorff continuity of $H$-codifferentials; however, one could reformulate all
results of this sections to the case of other types of continuity.

In the following propositions and theorems we mostly study weakly $H$-codifferentiable functions but all
results of this section are also valid for strongly $H$-codifferentiable functions.

\begin{remark}
(i) {We do not discuss any formulae for computing $H$-quasidifferentials, upper $H$-convex (lower $H$-concave)
approximations, and exhaustive families of these approximations. One can easily derive them arguing in a similar way to
the cases of exhaustive families of nonhomogeneous convex approximations \cite{Dolgopolik2} and quasidifferentiable
functions \cite{DemRub}.
}

\noindent(ii) {One can consider the main results of this section as sufficient conditions for the set of all
$H$-codifferentiable (or all continuously $H$-codifferentiable) at a given point functions to be a cone, a group under
addition (or multiplication), a linear space, an algebra, a lattice or a vector lattice.
}
\end{remark}

As earlier mentioned, we suppose that $X$ is a Hausdorff topological vector space over the field of real or complex
numbers, $E$ is an order complete Hausdorff topological vector lattice, $H$ is an arbitrary nonempty set of functions
$h \colon X \to \overline{E}$, and $\Omega \subset X$ is an open set. 

It is obvious that if a function $F \colon \Omega \to E$ is weakly $H$-codifferentiable at $x \in \Omega$, then for any
$c \in E$ the function $F + c$ is also weakly $H$-codifferentiable at $x$, $\delta (F + c)_H [x] = \delta F_H[x]$ and
$D^w_H (F + c)(x) = D^w_H F(x)$. It is easy to check that the following propositions hold true.

\begin{proposition}
Let a function $F \colon \Omega \to E$ be weakly $H$-codifferentiable at a point $x \in \Omega$, and let $\alpha \in
\mathbb{R}$ be arbitrary. Suppose also that $H$ is a cone in the case $\alpha \ne 0$, and $0 \in H$ in the case 
$\alpha = 0$. Then the function $\alpha F$ is weakly $H$-codifferentiable at the point $x$,
$\delta (\alpha F)_H[x] = \alpha \delta F_H [x]$ and $D^w_H(\alpha F)(x) = \alpha D^w_H F(x)$ in the case 
$\alpha \ge 0$, and the function $\alpha F$ is weakly $(-H)$-codifferentiable at $x$, 
$\delta (\alpha F)_{(-H)}[x] = \alpha \delta F_H [x]$ and $D^w_{(-H)}(\alpha F)(x) = \alpha D^w_H F(x)$ in the case 
$\alpha < 0$.
\end{proposition}

\begin{corollary}
Let all assumptions of the previous proposition be satisfied, and let $(H, d)$ be a metric space. Suppose that the
mapping $h \to \alpha h$ is uniformly continuous on $H$ (if $\alpha < 0$ and $(-H) \nsubseteq H$, then we suppose that 
the set $(-H)$ is equipped with a metric; in particular, one can suppose that $d(-h,-p) = d(h, p)$ for all 
$h, p \in H$). Suppose also that the function $F$ is Hausdorff continuously weakly
$H$-codifferentiable at the point $x$. Then the function $\alpha F$ is Hausdorff continuously weakly
$H$-codifferentiable at the point $x$.
\end{corollary}

\begin{proposition}
Let functions $F_1, F_2 \colon \Omega \to E$ be weakly $H$-codifferentiable at a point $x \in \Omega$, and let
the set $H$ be closed under addition. Then the function $F_1 + F_2$ is weakly $H$-codifferentiable at the point $x$,
$\delta (F_1 + F_2)_H [x] = \delta (F_1)_H[x] + \delta (F_2)_H[x]$ and 
$D^w_H (F_1 + F_2)(x) = D^w_H F_1(x) + D^w_H F_2(x)$.
\end{proposition}

\begin{corollary}
Let all assumption of the previous proposition be satisfied, and let $(H, d)$ be a metric space. Suppose that 
the mapping $(h, p) \to h + p$ is uniformly continuous on $H \times H$; in particular, one can suppose that there
exists $C > 0$ such that 
\[
d( h_1 + h_2, p_1 + p_2 ) \le C(d(h_1, p_1) + d(h_2, p_2)) \quad \forall h_1, h_2, p_1, p_2 \in H.
\]
Suppose also that the functions $F_1$ and $F_2$ are Hausdorff continuously weakly $H$-codifferentiable at the point
$x$. Then the function $F_1 + F_2$ is Hausdorff continuously weakly $H$-codifferentiable at $x$.
\end{corollary}

Let us study the problem of finding the $H$-codifferential of the superposition of functions.

\begin{theorem} \label{ThSuperpWithGateauxDiff}
Let $X$, $Y$ be arbitrary normed spaces, $E$ be an order complete normed lattice. Suppose that the following conditions
are satisfied:
\begin{enumerate}

\item{a function $F \colon \Omega \to E$ is strongly $H$-codifferentiable at a point $x \in \Omega$;}

\item{$F'_H[x]$ is Lipschitz continuous in a neighbourhood of zero.}

\item{$S \subset Y$ is an open set, $y \in S$ is arbitrary;}

\item{a function $G \colon S \to X$ is continuous and G\^ateaux differentiable at the point $y$, and $G(y) = x$.}

\end{enumerate}

Then there exists an open set $\mathcal{O} \subset S$ such that $y \in \mathcal{O}$, the function $T = F \circ G$ is
defined on $\mathcal{O}$ and  weakly $\widehat{H}$-codifferentiable at the point $y$, where
\[
\widehat{H} = \{ \hat{h} \colon Y \to \mathbb{R} \mid \hat{h} = h \circ \delta G[y], h \in H \}
\]
and $\delta G[y]$ is the G\^ateaux derivative of the function $G$ at the point $y$. Moreover, for any 
$(\Phi, \Psi) \in F'_H [x]$ and $(U, V) \in D^s_H F(x)$ one has
\begin{gather}
\delta T_{\widehat{H}} [y] = [ \Phi \circ \delta G[y], \Psi \circ \delta G[y] ]_w, \label{composHDeriv} \\
D^w_{\widehat{H}} T(y) = \left[ \left\{ \hat{h} = h \circ \delta G[y] \in \widehat{H} \bigm| h \in U \right\},
\left\{ \hat{p} = p \circ \delta G[y] \in \widehat{H} \bigm| p \in V \right\} \right]_w \label{composHCodiff}.
\end{gather}
\end{theorem}

\begin{proof}
Note that the right-hand sides of equalities (\ref{composHDeriv}) and (\ref{composHCodiff}) do not depend on the choice
of $(\Phi, \Psi) \in F'_H [x]$ and $(U, V) \in D^s_H F(x)$. Hence these formulae are correct.

From the facts that the function $G$ is continuous at the point $y$ and the set $\Omega$ is open it follows that there
exists $\mu > 0$ such that $\mathcal{O}(y, \mu) \subset S$ and $G( \mathcal{O}(y, \mu) ) \subset \Omega$. Denote 
$\mathcal{O} = \mathcal{O}(y, \mu)$. It is clear that the composition $F \circ G$ is defined at least on $\mathcal{O}$.

Fix an arbitrary $\Delta y \in \mathcal{O}(0, \mu)$. The function $G$ is G\^ateaux differentiable at the point $y$
hence 
\[
G(y + \Delta y) = G(y) + \delta G[y](\Delta y) + o_G(\Delta y),
\]
where $o_G(\alpha \Delta y) / \alpha \to 0$ as $\alpha \downarrow 0$. Denote 
$\omega(\Delta y) = \delta G[y](\Delta y) + o_G(\Delta y)$. It is obvious that there exists $\alpha_0 \in (0, 1)$ such
that for all $\alpha \in (0, \alpha_0)$
\[
\| \omega(\alpha \Delta y) \| \le \alpha (\| \delta G[y] \| + 1) \| \Delta y \|.
\]
In particular, one has that $\| \omega(\alpha \Delta y) \| \to 0$ as $\alpha \downarrow 0$.

Fix an arbitrary $(\Phi, \Psi) \in F'_H[x]$ such that the functions $\Phi(\cdot)$ and $\Psi(\cdot)$ are
Lipschitz continuous in a neighbourhood of zero (such $\Phi$ and $\Psi$ exist, since $F'_H[x]$ is Lipschitz
continuous in a neighbourhood of zero). The function $F$ is strongly $H$-codifferentiable at the point $x = G(y)$,
therefore for any admissible $\Delta x \in X$ one has
\[
F(x + \Delta x) - F(x) = \Phi(\Delta x) + \Psi(\Delta x) + o_F(\Delta x),
\]
where $\| o_F(\Delta x) \| / \| \Delta x \| \to 0$ as $\Delta x \to 0$ or, equivalently, 
$\| o_F(\Delta x) \| = \beta( \Delta x ) \| \Delta x \|$, where $\beta( \Delta x ) \to 0$ as $\Delta x \to 0$. Thus, one
gets
\begin{multline*}
T(y + \Delta y) - T(y) = F( G(y + \Delta y) ) - F(G(y)) = F( G(y) + \omega(\Delta y) ) - F(G(y)) = \\
= \Phi( \omega(\Delta y) ) + \Psi( \omega( \Delta y) ) + o_F( \omega( \Delta y ) ).
\end{multline*}
For any $\alpha \in (0, \alpha_0 )$ one has
\[
\| o_F( \omega( \alpha \Delta y ) ) \| \le \alpha \beta( \omega( \alpha \Delta y ) ) (\| G'[y] \| + 1)\| \Delta y \|.
\]
Observe that $\beta( \omega( \alpha \Delta y ) ) \to 0$ as $\alpha \downarrow 0$, since $\beta( \Delta x ) \to 0$ as
$\Delta x \to 0$ and $\| \omega(\alpha \Delta y) \| \to 0$ as $\alpha \downarrow 0$. Therefore
$ \| o_F( \omega( \alpha \Delta y ) ) \| / \alpha \to 0$ as $\alpha \downarrow 0$. It remains to note that from the
Lipschitz continuity of $\Phi(\cdot)$ and $\Psi(\cdot)$ in a neighbourhood of zero it follows that
\[
\Phi( \omega( \Delta y ) ) + \Psi( \omega( \Delta y ) ) = 
(\Phi \circ \delta G[y])(\Delta y) ) + (\Psi \circ \delta G[y])(\Delta y) ) + o( \Delta y ),
\]
where $o(\alpha \Delta y, y) / \alpha \to 0$ as $\alpha \downarrow 0$.
\end{proof}

Let us recall some definitions from lattice theory (see~\cite{Birkhoff, Schaefer, MeyerNieberg}). Let $Y$ be an order
complete vector lattice. Denote by $L(E, Y)_+$ the set of all positive linear operators mapping $E$ to $Y$. 
A linear operator $T \colon E \to Y$ is said to be a complete lattice homomorphism, if for any bounded from above set $A
\subset E$ one has $T \sup_{x \in A} x = \sup_{x \in A} T x$, and for any bounded from below set $B \subset E$ one has
$T \inf_{x \in B} x = \inf_{x \in B} T x$. It is clear that any complete lattice homomorphism $T$ is a positive
operator.

A linear operator $T \colon E \to Y$ is said to be completely regular, if there exist complete lattice homomorphisms 
$S, R \colon E \to Y$ such that $T = S - R$. One can verify that the representation $T = S - R$ of the completely
regular operator $T$ as the difference of two complete lattice homomorphisms is not unique. It is easy to see that any
linear mapping $T \colon \mathbb{R}^m \to \mathbb{R}^n$, where $\mathbb{R}^m$ and $\mathbb{R}^n$ are endowed with
the canonical order relations, is completely regular. 

\begin{theorem}
Let $X$ be a normed space, $E$ and $Y$ be order complete normed lattices. Suppose
that the following conditions are satisfied:
\begin{enumerate}

\item{the set $H$ is closed under addition and for any $h \in H$ one has $-h \in H$;}

\item{a function $F \colon \Omega \to E$ is weakly $H$-codifferentiable at a point $x$;}

\item{$\delta F_H[x]$ is Lipschitz continuous in a neighbourhood of zero;}

\item{$\Sigma \subset E$ is an open set such that $F(x) \in \Sigma$;}

\item{a function $G \colon \Sigma \to Y$ is Fr\'echet differentiable at a point $F(x)$;}

\item{the Fr\'echet derivative $G'[F(x)]$ of $G$ at $F(x)$ is a completely regular linear mapping;}

\item{the function $T = G \circ F$ is defined on an open set $\mathcal{O} \subset \Omega$ (in particular, one can
suppose that $\Sigma = E$ or that $F$ is continuous at $x$).}

\end{enumerate}
Then $T$ is weakly $\widehat{H}$-codifferentiable at a point $x$, where 
$\widehat{H} = \{ S \circ h \mid h \in H, S \in L(E, Y)_+ \}$. Futhermore, for all complete lattice homomorphisms 
$S, R \colon E \to Y$ such that $G'[F(x)] = S - R$, and for any $(\Phi, \Psi) \in \delta F_H[x]$ and 
$(U, V) \in D^w_H F(x)$ one has
\begin{gather} \label{ComposWithComplRegMap1}
\delta T_{\widehat{H}} [x] = [ S \circ \Phi - R \circ \Psi, S \circ \Psi - R \circ \Phi ]_w, \\
D^w_{\widehat{H}} T(x) = [ \{ S \circ h - R \circ p \mid h \in U, p \in V \},
\{ S \circ p - R \circ h \mid h \in U, p \in V \}]. \label{ComposWithComplRegMap2}
\end{gather}
\end{theorem}

\begin{proof}
Fix arbitrary complete lattice homomorphisms $S, R \colon E \to Y$ such that $G'[F(x)] = S - R$, and an arbitrary
$(\Phi, \Psi) \in \delta F_H [x]$. Arguing in a similar way to the proof of theorem \ref{ThSuperpWithGateauxDiff} one
can show that for any admissible argument increment $\Delta x \in X$
\[
T(x + \Delta x) - T(x) = G'[F(x)]( \Phi(\Delta x) + \Psi(\Delta x) ) + o( \Delta x, x),
\]
where $o(\alpha \Delta x, x) / \alpha \to 0$ as $\alpha \downarrow 0$. Let us show that the function 
$K(\cdot) = G'[F(x)]( \Phi(\cdot) + \Psi(\cdot) )$ can be represented as the sum of $\widehat{H}$-convex and
$\widehat{H}$-concave functions. Indeed, let $\Phi$ be generated by $U \subset H$, and $\Psi$ be generated by $V
\subset H$. Hence for any $x \in X$
\begin{multline*}
K(x) = (S - R)(\sup_{h \in U}h(x) + \inf_{p \in V}p(x)) = \\
= \sup_{h \in U, p \in V}((S + R)(h(x) - p(x)) + \inf_{h \in U, p \in V} ((S + R)(-h(x) + p(x))).
\end{multline*}
Taking into account the assumptions about the set $H$, and the fact that the sum of two complete lattice homomorphisms
is a positive linear operator one gets that $K$ is the sum of $\widehat{H}$-convex and $\widehat{H}$-concave
functions.

It remains to note that formulae (\ref{ComposWithComplRegMap1})-(\ref{ComposWithComplRegMap2}) do not depend on the
choice of complete lattice homomorphisms $S, R \colon E \to Y$, such that $G'[F(x)] = S - R$, 
$(\Phi, \Psi) \in \delta F_H[x]$ and $(U, V) \in D^w_H F(x)$, since 
\[
(S \circ \Phi - R \circ \Psi) + (S \circ \Psi - R \circ \Phi) = G'[F(x)](\Phi + \Psi), 
\]
and for any function $w \colon X \to E$ and $x \in X$ one has $G'[F(x)](w(\alpha x)) / \alpha \to 0$ as 
$\alpha \downarrow 0$, whenever $w(\alpha x) / \alpha \to 0$ as $\alpha \downarrow 0$.
\end{proof}

As a simple, yet useful corollary to the previous proposition one gets the following result.

\begin{theorem} \label{superpositionInRd}
Let $X$ be an arbitrary normed space. Suppose that the following conditions are satisfied:
\begin{enumerate}

\item{the set $H$ is a linear subspace of $\mathbb{R}^X$ (where $\mathbb{R}^X$ is the set of all functions mapping $X$
to $\mathbb{R}$);}

\item{functions $F_i \colon \Omega \to \mathbb{R}$ are weakly $H$-codifferentiable at a point $x \in \Omega$, 
$i \in I = \{1, \ldots, d\}$;}

\item{$\delta (F_i)_H[x]$ are Lipschitz continuous in a neighbourhood of zero, $i \in I$;}

\item{$S \subset \mathbb{R}^d$ is an open set such that $y = (F_1(x), \ldots, F_d(x)) \in S$;}

\item{a function $g \colon S \to \mathbb{R}$ is differentiable at the point $y$;}

\item{the function $T(\cdot) = g(F_1(\cdot), \ldots, F_d(\cdot))$ is defined on an open set 
$\mathcal{O} \subset \Omega$.}

\end{enumerate}
Then the function $T$ is weakly $H$-codifferentiable at the point $x$ and
\[
\delta T_H [x] = \sum_{i \in I} \frac{\partial g}{\partial y_i}(y) \delta (F_i)_H[x], \quad
D^w_H T(x) = \sum_{i \in I} \frac{\partial g}{\partial y_i}(y) D^w_H F_i(x).
\]
\end{theorem}

\begin{corollary}
Let all assumptions of the previous theorem be satisfied, and let $(H, d)$ be a metric space. Suppose that the mapping
$(\alpha, h) \to \alpha h$, $\alpha \in \mathbb{R}$ is uniformly continuous on $\mathbb{R} \times H$, and the mapping
$(h, p) \to h + p$ is uniformly continuous on $H \times H$ (in particular, one can suppose that $d$ is a norm). Suppose
also that all functions $F_i$ are continuous and Hausdorff continuously weakly $H$-codifferentiable at the point $x$,
and the function $g$ is continuously differentiable at the point $y$. Then the function $T$ is Hausdorff continuously
weakly $H$-codifferentiable at $x$.
\end{corollary}

\begin{remark}
Let functions $F, F_1, F_2 \colon \Omega \to \mathbb{R}$ be (continuously) $H$-codifferentiable at a point 
$x \in \Omega$, and let $F \ne 0$ in a neighbourhood of $x$. As simple corollaries to theorem \ref{superpositionInRd}
one gets the $H$-codifferentiability (continuous $H$-codifferentiability) of the functions $F_1 \cdot F_2$ and
$1 / F$ at the point $x$ under suitable assumptions on the set $H$.
\end{remark}

Let us consider the supremum and the infimum of $H$-codifferentiable functions.

\begin{theorem} \label{TheoremSupInf}
Let functions $F_i \colon \Omega \to E$ be weakly $H$-codifferentiable at a point $x \in \Omega$, 
$i \in I = \{ 1, \ldots, n \}$. Suppose that the set $H$ satisfies the following assumptions:
\begin{enumerate}
\item{$H$ is closed under addition;}

\item{for any $h \in H$ one has $-h \in H$;}

\item{$H$ is closed under vertical shifts, i.~e. for any $c \in E$, $h \in H$ one has $h + c \in H$.}
\end{enumerate}
Then the functions $F = \sup_{i \in I} F_i$ and $G = \inf_{i \in I} F_i$ are weakly $H$-codifferentiable at the point
$x$. Moreover, for any $(\Phi_i, \Psi_i) \in \delta (F_i)_H[x]$ and $(U_i, V_i) \in D^w_H F_i(x)$, $i \in I$, one has
\begin{gather}
\delta F_H[x] = \left[ \sup_{i \in I} \bigg( F_i(x) - F(x) + \Phi_i - \sum_{j \in I \setminus \{ i \}} \Psi_j \bigg), 
\sum_{k \in I} \Psi_k \right]_w, \label{HderivOfSup} \\
\delta G_H[x] = \left[ \sum_{k \in I} \Phi_k, 
\inf_{i \in I} \bigg( F_i(x) - G(x) + \Psi_i - \sum_{j \in I \setminus \{ i \}} \Phi_j \bigg), \right]_w,
\label{HderivOfInf}
\end{gather}
and
\begin{gather}
D^w_H F(x) = \left[ \bigcup_{i \in I} \bigg\{ \{F_i(x) - F(x)\} + U_i - \sum_{j \in I \setminus \{ i \}} V_i \bigg\},
\sum_{k \in I} V_i \right]_w, \label{HcodiffOfSup} \\
D^w_H G(x) = \left[ \sum_{k \in I} U_i,
\bigcup_{i \in I}\bigg\{ \{F_i(x) - G(x)\} + V_i - \sum_{j \in I \setminus \{ i \} } U_i \bigg\} \right]_w.
\label{HcodiffOfInf}
\end{gather}
\end{theorem}

\begin{proof}
Note that the right-hand sides of formulae (\ref{HderivOfSup})-(\ref{HcodiffOfInf}) do not depend on the choice of
$(\Phi_i, \Psi_i) \in \delta (F_i)_H[x]$ and $(U_i, V_i) \in D^w_H F_i(x)$. Therefore these formulae are correct.

We only consider the function $F$, since the assertion about the function $G$ is proved in a similar way. Fix arbitrary
$(\Phi_i, \Psi_i) \in \delta F_H[x]$ and $(U_i, V_i) \in D^w_H F(x)$, $i \in I$. For any admissible
argument increment $\Delta x \in X$ one has
\[
F_i(x + \Delta x) = F_i(x) + \Phi_i(\Delta x) + \Psi_i(\Delta x) + o_i(\Delta x, x),
\]
where $o_i(\alpha \Delta x, x) / \alpha \to 0$ as $\alpha \downarrow 0$. Hence
$$
F(x + \Delta x) - F(x) = \sup_{i \in I}(F_i(x) - F(x) + \Phi_i(\Delta x) + \Psi_i(\Delta x) + o_i(\Delta x, x)).
$$
Applying the simple inequality
\begin{multline*}
\big| \sup_{i \in I}(F_i(x) - F(x) + \Phi_i(\Delta x) + \Psi_i(\Delta x) + o_i(\Delta x, x)) - \\
- \sup_{i \in I}(F_i(x) - F(x) + \Phi_i(\Delta x) + \Psi_i(\Delta x)) \big| \le \sum_{i \in I} |o_i(\Delta x, x)|
\end{multline*}
one gets
$$
F(x + \Delta x) - F(x) = \sup_{i \in I}(F_i(x) - F(x) + \Phi_i(\Delta x) + \Psi_i(\Delta x)) + o(\Delta x, x),
$$
where $o(\alpha \Delta x, x) / \alpha \to 0$ as $\alpha \downarrow 0$. It remains to note that
\begin{multline*}
\sup_{i \in I}(F_i(x) - F(x) + \Phi_i(\Delta x) + \Psi_i(\Delta x)) = \\
= \sup_{i \in I}\bigg(F_i(x) - F(x) + \Phi_i(\Delta x) - \sum_{j \in I \setminus\{j\}} \Psi_j(\Delta x)\bigg) + 
\sum_{k = 1}^n \Psi_k(\Delta x),
\end{multline*}
and the fact that the right-hand side of the last equality is the sum of $H$-convex and $H$-concave functions.
\end{proof}

\begin{corollary} \label{CorContOfSupInf}
Suppose that all assumptions of the previous theorem are satisfied and $0 \in H$ (or, equivalently, for any $c \in E$
the function $h \equiv c$ belongs to $H$). Denote by $\ell \colon E \to H$ the natural embedding of $E$ in $H$, i.~e.
$(\ell(y))(\cdot) \equiv y$ for all $y \in E$. Let $(H, d)$ be a metric space such that the following assumptions are
satisfied:
\begin{enumerate}

\item{the mapping $(h, p) \to h + p$ is uniformly continuous on $H \times H$;}

\item{the mapping $h \to -h$ is uniformly continuous on $H$;}

\item{the quotient topology on $\ell(E)$ induced by $\ell$ is finer than the topology induced by the metric $d$
}

\end{enumerate}
(therefore if a function $T \colon \Omega \to E$ is continuous, then the function $\ell \circ T \colon \Omega \to H$ is
also continuous). Suppose also that all functions $F_i$ are continuous and Hausdorff continuously weakly
$H$-codifferentiable at the point $x$, $i \in I$. Then the functions $F$ and $G$ are Hausdorff continuously weakly
$H$-codifferentiable at $x$.
\end{corollary}

\begin{remark}
One can easily prove that under suitable assumptions the supremum of an infinite family of weakly
$H$-hypodifferentiable functions is also weakly $H$-hypodifferentiable (and that the infimum of an infinite family of
weakly $H$-hyperdifferentiable functions is weakly $H$-hypodifferentiable).
\end{remark}

\section{Necessary optimality conditions}\label{SectionNessOptCond}

In this section we derive necessary optimality conditions for $H$-quasidifferentiable and $H$-codifferentiable
functions with the use of upper abstract convex and lower abstract concave approximations. Then we show how they can be
transformed into more constructive necessary optimality conditions in some particular cases.

\subsection{General necessary conditions for an extremum}

In this section we only consider the case $E = \mathbb{R}$; however, one can modify main results of this section to the
case of general order complete topological vector lattices.

We need an auxiliary definition (see~\cite{Zaffaroni, RubinovZaff}).

\begin{definition}
Let $f \colon X \to \overline{\mathbb{R}}$ be an arbitrary function such that $f(0) = 0$. The function $f$ is said to be
subhomogeneous (superhomogeneous) if for any $\Delta x \in X$ and $\alpha \in (0, 1)$ one has
\[
f(\alpha \Delta x) \le \alpha f(\Delta x) \quad (f(\alpha \Delta x) \ge \alpha f(\Delta x)).
\]
\end{definition}

The class of all subhomogeneous (or superhomogeheous) functions is very broad. In particular, any convex (concave)
function $f \colon X \to \overline{\mathbb{R}}$ such that $f(0) = 0$ is subhomogeneous (superhomogeneous). Also, any
positively homogeneous of degree $\lambda \ge 1$ ($\lambda \in (0, 1]$) function is subhomogeneous (superhomogeneous).

Let $A \subset X$ be a convex set, and let the set $H$ be closed under vertical shifts. If $x \in X$ then denote 
$A - x = \{ y \in X \mid y = a - x, \; a \in A \}$. Consider the following optimization problem
\begin{equation} \label{mathProgInf}
f_0(x) \to \inf, \quad x \in A, \quad f_i(x) \le 0, \quad i \in I,
\end{equation}
where $f_i \colon X \to \mathbb{R}$, $i \in I_0 = \{ 0 \} \cup I$, $I = \{ 1, \ldots, n \}$.

\begin{theorem}
Let functions $\varphi_i \colon X \to \overline{\mathbb{R}}$ be weak upper $H$-convex approximations of the functions
$f_i$ at a point $x^* \in A$ such that $\varphi_i(0) = 0$, $i \in I_0$. Suppose that $x^*$ is a point of local minimum
of problem (\ref{mathProgInf}), and the $H$-convex function
\begin{equation} \label{mathProgUCA}
g(\cdot) = \sup\{ \varphi_0(\cdot), \varphi_1(\cdot) + f_1(x^*), \ldots, \varphi_n(\cdot) + f_n(x^*) \}
\end{equation}
is subhomogeneous. Then $0$ is a point of global minimum of the function $g$ on the set $A - x^*$. Moreover, if $A = X$
and $0 \in H$, then $0 \in \underline{\partial}_H g(0)$.
\end{theorem}

\begin{proof}
From the fact that $x^*$ is a point of local minimum of problem (\ref{mathProgInf}) it follows that $x^*$ is a
point of local minimum of the function
\[
F(\cdot) = \max\{ f_0(\cdot) - f_0(x^*), f_1(\cdot), \ldots, f_n(\cdot) \}
\]
on the set $A$. It is easy to check that the function $g$ (see.~(\ref{mathProgUCA})) is a weak upper $H$-convex
approximation of the function $F$ at $x^*$ and $g(0) = 0$.   
                                                                                                              
Suppose that $0$ is not a point of global minimum of the function $g$ on the set $A - x^*$. Then there exists 
$y \in A$ such that $g(y - x^*) = - m < 0 = g(0)$. Denote $\Delta x = y - x^*$. Since $g$ is a weak upper $H$-convex
approximation of the function $F$ at the point $x^*$ and $g$ is subhomogeneous, there exists $\delta \in (0, 1)$ such
that
\[
F(x^* + \alpha \Delta x) - F(x^*) \le g(\alpha \Delta x) + \frac{m}{2} \alpha
\le \alpha g(\Delta x) + \frac{m}{2} \alpha = -\frac{m}{2} \alpha \quad \forall \alpha \in (0, \delta),
\]
which contradicts the fact that $x^*$ is a point of local minimum of $F$ on $A$.
\end{proof}

Arguing in a similar way one can prove the following theorem, which is the ``mirror version'' of the previous one.

\begin{theorem} \label{ThMathProgrMax}
Let functions $\psi_i \colon X \to \overline{\mathbb{R}}$ be weak lower $H$-concave approximations of the functions
$f_i$ at a point $x^* \in A$ such that $\psi_i(0) = 0$, $i \in I_0$. Suppose that $x^*$ is a point of local maximum in
the problem
\begin{equation} \label{mathProgSup}
f_0(x) \to \sup, \quad x \in A, \quad f_i(x) \ge 0, \quad i \in I, 
\end{equation}
and the $H$-concave function
\[
g(\cdot) = \inf\{ \psi_0(\cdot), \psi_1(\cdot) + f_1(x^*), \ldots, \psi_n(\cdot) + f_n(x^*) \}
\]
is superhomogeneous. Then $0$ is a point of global maximum of the function $g$ on the set $A - x^*$. Moreover, if 
$A = X$ and $0 \in H$, then $0 \in \overline{\partial}_H g(0)$.
\end{theorem}

One can obtain necessary optimality conditions in terms of abstract convex approximations for more general optimization
problems, although it requires more restrictive assumptions. Namely, let $X$ be a normed space and $M \subset \Omega$
be a nonempty set. For any $x \in \cl M$ denote by $T_M(x)$ the contingent cone to the set $M$ at the point $x$ (see
\cite{AubinFrankowska}, chapter 4). The following theorem holds true.

\begin{theorem} \label{ThGenMathProg}
Let $x^* \in X$ be a point of local minimum in the problem
\begin{equation} \label{MathProgGenUCA}
f_0(x) \to \inf, \quad x \in M, \quad f_i(x) \le 0, \quad i \in I,
\end{equation}
Suppose that a function $\varphi_i \colon X \to \overline{\mathbb{R}}$ is a strong upper $H$-convex approximation of
the function $f_i$ at the point $x^*$ such that $\varphi_i(0) = 0$, and $\varphi_i$ is Lipschitz continuous in a
neighbourhood of zero, $i \in I_0$. Suppose also that the $H$-convex function
\begin{equation} \label{UCAofNewObF}
g(x) = \sup\{ \varphi_0(x), \varphi_1(x) + f_1(x^*), \ldots, \varphi_n(x) + f_n(x^*) \} \quad x \in X
\end{equation}
is subhomogeneous. Then $0$ is a point of global minimum of $g$ on $T_M(x^*)$.
\end{theorem}

\begin{proof}
It is clear that $x^*$ is a point of local minimum of the function 
\[
F(\cdot) = \max\{ f_0(\cdot) - f(x^*), f_1(\cdot), \ldots, f_n(\cdot) \}
\]
on the set $M$. Also, it is easy to verify that the function $g$ (see~(\ref{UCAofNewObF})) is a strong upper
$H$-convex approximation of $F$ at $x^*$. Moreover, $g$ is Lipschitz continuous in a neighbourhood of zero 
and $g(0) = 0$.

Suppose that there exists $v \in T_M(x^*)$ such that $g(v) = -m < g(0)$. By the definition of $T_M(x^*)$ there exist
sequences $\{ h_n \} \subset (0, +\infty)$ and $\{ v_n \} \subset X$ such that $x^* + h_n v_n \in M$, $h_n \downarrow 0$
and $v_n \to v$ as $n \to \infty$.

Applying the fact that $g$ is a strong upper $H$-convex approximation of $F$ at $x^*$ one has that there exist $r > 0$
and a function $\beta \colon B(0, r) \to \mathbb{R}$ such that $\beta(\Delta x) \to 0$ as $\Delta x \to 0$ and
\[
F(x^* + \Delta x) - F(x^*) \le g(\Delta x) + \beta(\Delta x) \| \Delta x \| \quad \forall \Delta x \in B(0, r).
\]
Hence there exists $n_1 \in \mathbb{N}$ such that for any $n > n_1$ one has 
$| \beta(h_n v_n ) | \| v_n \| \le m / 3$. Since $g$ is Lipschitz continuous in a neighbourhood of zero and $v_n \to v$,
there exist $L > 0$ and $n_2 \in \mathbb{N}$ such that for all $n > n_2$
\[
| g(h_n v_n) - g(h_n v) | \le L h_n \| v_n - v \| \le \frac{m}{3} h_n.
\]
Therefore, taking into account the subhomogeneity of $g$, one gets that for any $n > \max\{ n_1, n_2 \}$
\begin{multline*}
F(x^* + h_n v_n) - F(x^*) \le g(h_n v_n) + \beta(h_n v_n) h_n \| v_n \| \le \\
\le g(h_n v) + \frac{2m}{3} h_n \le -m h_n + \frac{2m}{3} h_n < 0,
\end{multline*}
which contradicts the fact that $x^*$ is a point of local minimum of $F$ on $M$.
\end{proof}

\begin{remark}
One can easily proof an analogous theorem about necessary condition for a local maximum in the problem
\[
f_0(x) \to \sup, \quad x \in M, \quad f_i(x) \ge 0, \quad i \in I
\]
in terms of strong lower $H$-concave approximations.
\end{remark}

As obvious corollaries to the previous theorems one gets the following necessary optimality conditions for
$H$-codifferentiable functions.

\begin{theorem} \label{ThNessMinCondForHcodiff}
Let the functions $f_i$, $i \in I_0$ be weakly $H$-codifferentiable at a point $x^* \in A$, and let $x^*$ be a point
of local minimum of problem (\ref{mathProgInf}). Suppose that the set $H$ is closed under addition and for any $h \in
H$ one has $0 \in \interior \dom h$. Then for any $(\Phi_i, \Psi_i) \in \delta (f_i)_H(x^*)$ and 
$p_i \in \overline{\partial}_H \Psi_i(0)$, $i \in I_0$ such that the $H$-convex function
\[
g(\cdot) = \sup\{ \Phi_0(\cdot) + p_0(\cdot), \Phi_1(\cdot) + p_1(\cdot) + f_1(x^*), \ldots,
\Phi_n(\cdot) + p_n(\cdot) + f_n(x^*) \}
\]
is subhomogeneous the function $g$ attains a global minimum on the set $A - x^*$ at the origin.
\end{theorem}

\begin{theorem}
Let the functions $f_i$, $i \in I_0$ and the set $H$ be as in the previous theorem. Suppose that $x^*$ is a point of
local maximum of problem (\ref{mathProgSup}). Then for any $(\Phi_i, \Psi_i) \in \delta (f_i)_H(x)$ and 
$h_i \in \underline{\partial}_H \Phi_i(0)$, $i \in I_0$ such that the $H$-concave function
\[
g(\cdot) = \inf\{ h_0(\cdot) + \Psi_0(\cdot), h_1(\cdot) + \Psi_1(\cdot) + f_1(x^*), \ldots,
h_n(\cdot) + \Psi_n(\cdot) + f_n(x^*) \}
\]
is superhomogeneous the function $g$ has a global maximum value on the set $A - x^*$ at the origin.
\end{theorem}

In general, upper abstract convex approximations are more convenient for the study of minimization problems, whereas
lower abstract concave approximations are more convenient for the study of maximization problems. Necessary conditions
for a maximum can be expressed in terms of upper abstract convex approximations, although these conditions are much
more cumbersome than the ones stated in theorem \ref{ThMathProgrMax}. 

We need additional notation. Denote
\[
\gamma(x, A) =  \{ g \in X \mid \exists \alpha > 0 \colon x + \alpha g \in A \}
\]
and $\Gamma(x, A) = \cl \gamma(x, A)$. It is easy to see that both $\gamma(x, A)$ and $\Gamma(x, A)$ are nonempty convex
cones.

\begin{theorem} \label{ThNessMaxCondExFamUCA}
Let $\{ \varphi_{\lambda} \}$, $\lambda \in \Lambda$ be an exhaustive family of weak upper $H$-convex approximations
of the function $f_0$ at a point $x^* \in A$, and let $x^*$ be a point of local maximum of the function $f_0$ on the set
$A$. Suppose that for any $g \in \Gamma(x^*, A)$ there exists $\alpha_g > 0$ such that for any $\lambda \in \Lambda$ the
function $\alpha \to \varphi_{\lambda}(\alpha g)$, $\alpha \in [0, \alpha_g)$ is convex. Then for any $\varepsilon > 0$
and $g \in \gamma(x^*, A)$ there exists $\lambda \in \Lambda$ such that
\begin{equation} \label{NessMaxCondUCA}
\varphi'_{\lambda}(0, g) \le \varepsilon,
\end{equation}
where $\varphi'_{\lambda}(0, g)$ is the directional derivative of the function $\varphi_{\lambda}$ at the
origin in the direction $g$. Moreover, if $\Lambda$ is finite, then for any $g \in \gamma(x^*, A)$ there exists 
$\lambda \in \Lambda$  such that
\[
\varphi'_{\lambda}(0, g) \le 0.
\]
If, in addition, for any $\lambda \in \Lambda$ the function $\varphi'_{\lambda}(0, \cdot)$ is continuous on
$\Gamma(x^*, A)$, then for any $g \in \Gamma(x^*, A)$ there exists $\lambda \in \Lambda$ such that the last inequality
holds true.
\end{theorem}

\begin{proof}
As it is well-known, from the convexity of the the function $z_{\lambda, g}(\alpha) = \varphi_{\lambda}(\alpha g)$, 
$\alpha \in [0, \alpha_g)$, $g \in \Gamma(x^*, A)$ it follows that there exists the right
derivative $(z_{\lambda, g})'_+(0)$ and the following equalities hold true
\begin{equation} \label{dirDerivConvFunc}
(z_{\lambda, g})'_+(0) = \lim_{\alpha \downarrow 0} \frac{\varphi(\alpha g) - \varphi(0)}{\alpha} =
\varphi'_{\lambda}(0, g) =
\inf_{\alpha \in (0, \alpha_g)} \frac{\varphi_{\lambda}(\alpha g) - \varphi_{\lambda}(0)}{\alpha}
\end{equation}
(see, e.g.,~\cite{IoffeTihomirov}, proposition 4.1.3).

Suppose that there exists $\varepsilon > 0$ and $g \in \gamma(x^*, A)$ such that inequality (\ref{NessMaxCondUCA})
does not hold true for any $\lambda \in \Lambda$. Applying (\ref{dirDerivConvFunc}) one gets
\begin{equation} \label{UCAnotMax}
\varphi_{\lambda}(\alpha g) \ge \varphi_{\lambda}(0) + \varepsilon \alpha \quad 
\forall \alpha \in [0, \alpha_g) \quad \forall \lambda \in \Lambda.
\end{equation}
Taking into account the facts that $g \in \gamma(x^*, A)$ and the set $A$ is convex, one can suppose 
that $\co \{x^*, x^* + \alpha_g g\} \subset A$.

By the definition of exhaustive family of weak upper $H$-convex approximations one has
$\inf_{\lambda \in \Lambda} \varphi_{\lambda}(0) = 0$, and there exists $\delta > 0$ such that
\[
f_0(x^* + \alpha g) - f_0(x^*) \ge \inf_{\lambda \in \Lambda} \varphi_{\lambda}(\alpha g) - 
\frac{\varepsilon}{2} \alpha \quad \forall \alpha \in (0, \delta).
\]
Thus, taking into account (\ref{UCAnotMax}) one has
\[
f_0(x^* + \alpha g) - f_0(x^*) \ge \inf_{\lambda \in \Lambda} \varphi_{\lambda}(0) + \frac{\varepsilon}{2} \alpha = 
\frac{\varepsilon}{2} \alpha \quad \forall \alpha \in (0, \min\{ \delta, \alpha_g \}),
\]
which contradicts the fact that $x^*$ is a point of local maximum of $f_0$ on $A$.
\end{proof}

\begin{remark}
(i) {An analogous theorem about necessary conditions for a minimum in terms of weak lower abstract concave
approximations also holds true.
}

\noindent(ii) {One can construct a numerical method for finding stationary points of an $H$-codifferentiable function
(as well as a numerical method for finding a solution of the equation $F(x) = 0$, where $F$ is $H$-codifferentiable)
based on the method for the search of a local minimizer of a nonsmooth function having a continuous approximation
(see~\cite{RubinovZaff, Zaffaroni}).
}
\end{remark}

\subsection{Necessary conditions for an extremum of abstract quasidifferentiable function}

Let us consider necessary optimality conditions for $H$-quasidifferentiable functions. We only discuss necessary
conditions for a minimum, since necessary conditions for a maximum are symmetrical to them. All necessary optimality
conditions stated below immediately follows from the necessary conditions for an extremum of a directionally
differentiable function. Therefore we omit the proofs.

Let all functions $h \in H$ be p.h., and, as earlier, suppose that $f_0 \colon \Omega \to \mathbb{R}$ is an arbitrary
function, $A \subset \Omega$ is a nonempty convex set.

\begin{theorem} \label{ThHquasidMinCond}
Let the function $f_0$ be Dini (Hadamard) $H$-quasidifferentiable at a point $x^* \in A$. Suppose that $x^*$ is a
point of local minimum of the function $f_0$ on the set $A$. Then for any $(\Phi, \Psi) \in \mathcal{D}_H f_0(x^*)$ and
for all $p \in \supp^+(\Psi, H)$ the function $\Phi + p$ attains a global minimum value on the set $\gamma(x^*, A)$ 
($\Gamma(x^*, A)$) at the origin. Also, for any $U \subset H$ such that $\Phi = \sup_{h \in U} h$,
for any $\varepsilon > 0$ and for all $g \in \gamma(x^*, A)$ ($g \in \Gamma(x^*, A)$) there exists $h \in U$ such that
$h(g) + \Psi(g) \ge - \varepsilon$. Moreover, if there exists $U \subset H$ such that
\begin{enumerate}
\item{$\Phi$ is generated by $U$,}
\item{for any $x \in X$ there exists $h \in U$ such that $\Phi(x) = h(x)$ (in particular, if $U$ is finite)}
\end{enumerate}
then for any $g \in \gamma(x^*, A)$ ($g \in \Gamma(x^*, A)$) there exists $h \in U$ such that $h(g) + \Psi(g) \ge 0$.
\end{theorem}

\begin{corollary}
Suppose that all assumption of the previous theorem are satisfied, $x^* \in \interior A$, and let $0 \in H$. Then for
any $(\Phi, \Psi) \in \mathcal{D}_H f_0(x^*)$ and for all $p \in \supp^+(\Psi, H)$ one has 
$0 \in \underline{\partial}_H (\Phi + p)(0)$.
\end{corollary}

\begin{remark}
Applying theorems \ref{ThNessMinCondForHcodiff} and \ref{ThHquasidMinCond} for different paricular sets $H$ one can
easily obtain well-known necessary optimality conditions for codifferentiable and quasidifferentiable
functions, and for functions having upper (lower) exhauster or upper (lower) coexhauster 
\cite{DemRub, Demyanov, Abankin, DemAbbasov}.
\end{remark}

\subsection{Some particular cases}\label{PartOptimCond}

Let us consider how general necessary optimality conditions for $H$-codifferentiable functions can be easily transformed
into more convenient conditions in some particular cases. In this subsection $X$ is a real Banach space, 
$E = \mathbb{R}$, $A \subset \Omega$ is a nonvoid closed convex set. Note, that if the set $H$ is closed under vertical
shifts then, without loss of generality, we may assume that for any weakly $H$-codifferentiable function $f$
and for all $(\Phi, \Psi) \in \delta f_H$ one has $\Phi(0) = \Psi(0) = 0$.

Let $f, f_i \colon X \to \mathbb{R}$ be arbitrary functions, $i \in I_0 = \{ 0 \} \cup I$, where 
$I = \{ 1, \ldots, n \}$. For any $x \in X$ denote $R(x) = \{ 0 \} \cup \{ i \in I \mid f_i(x) = 0 \}$.

\begin{example}
Let $H$ coincide with the set of all continuous affine function $h \colon X \to \mathbb{R}$. Then, as it was shown in
example \ref{ExampleCodiff}, the function $f$ is weakly $H$-codifferentiable at a point $x \in \Omega$ iff $f$ is
codifferentiable at this point.

Let us derive necessary optimality conditions for a codifferentiable function in the problem with smooth
equality and codifferentiable inequality constraints.

\begin{proposition} \label{PrGenNCMCodiff}
Let $Y$ be a Banach space, a mapping $F \colon X \to Y$ be continuously Fr\'echet differentiable at a point $x^* \in X$,
the functions $f_i$ be Fr\'echet codifferentiable at a point $x^*$, $i \in I_0$. Suppose that the Fr\'echet derivative
$F'[x^*]$ of the map $F$ at $x^*$ is surjective, and $x^*$ is a point of local minimum in the problem
\[
f_0(x) \to \inf, \quad F(x) = 0, \quad f_i(x) \le 0, \quad i \in I.
\]
Then for any $(0, q_i) \in \overline{d} f_i(x^*)$, $i \in R(x^*)$ there exists $y^* \in Y^*$ such that
\[
(0, y^* \circ F'[x^*]) \in \Big( \co \bigcup_{i \in R(x^*)} (\underline{d} f_i(x^*) + \{ (0, q_i) \}) \Big).
\]
\end{proposition}

\begin{proof}
Fix an arbitrary $(0, q_i) \in \overline{d} f_i(x^*)$, $i \in I_0$ and define
\[
\varphi_i(x) = \max_{(a, p) \in \underline{d} f_i(x^*) + (0, q_i)} ( a + p(x) ) \quad \forall x \in X, 
\forall i \in I_0.
\]
It is easy to check that the function $\varphi_i$ is a strong upper $H$-convex approximation of $f_i$ at $x^*$,
$\varphi_i(0) = 0$ and $\varphi_i$ is Lipschitz continuous in a neighbourhood of zero 
(\cite{Zalinescu}, corollary 2.2.12), $i \in I_0$. 

Denote $M = \{ x \in X \mid F(x) = 0 \}$. By virtue of theorem \ref{ThGenMathProg} one has that $0$ is a point of
global minimum of the convex function
\[
g(\cdot) = \max\{ \varphi_0(\cdot), \varphi_1(\cdot) + f_1(x^*), \ldots, \varphi_n(\cdot) + f_n(x^*) \}
\]
on the set $T_M (x^*)$. Taking into account the Lusternik theorem (see~\cite{IoffeTihomirov}, section 0.2) one has that
$\Ker F'[x^*] \subset T_M(x^*)$, where $\Ker F'[x^*]$ is the kernel of the linear operator $F'[x^*]$. Therefore,
applying the necessary and sufficient condition for a minimum of a convex function on a closed convex set 
(\cite{IoffeTihomirov}, theorem 1.1.2${}'$) and the theorem about the subdifferential of the maximum of a finite family
of convex functions (\cite{Zalinescu}, corollary 2.8.11), one gets
\[
\underline{\partial} g(x^*) \cap ( - N(0, \Ker F'[x^*])) \ne \emptyset, \quad 
\underline{\partial} g(x^*) = \co \bigcup_{i \in R(x^*)} \underline{\partial} g_i(x^*),
\]
where $\underline{\partial} g(x^*)$ is the subdifferential of the convex function $g$ at $x^*$ and 
$N(0, \Ker F'[x^*]) = \{ p \in X^* \mid p(x) \le 0 \; \forall x \in \Ker F'[x^*] \}$ is the normal cone to
the set $\Ker F'[x^*]$ at the point $0$. By virtue of the theorem about the subdifferential of the supremum
(\cite{IoffeTihomirov}, theorem 4.2.3) one has 
$\{ 0 \} \times \underline{\partial} g_i(x^*) \subset \underline{d} f_i(x^*) + \{ (0, q_i) \}$. Hence
\[
\Big( \co \bigcup_{i \in R(x^*)} (\underline{d} f_i(x^*) + \{ (0, q_i) \}) \Big) \cap 
\big( \{ 0 \} \times (- N(0, \Ker F'[x^*])) \big) \ne \emptyset
\]
It remains to note that $N(0, \Ker F'[x^*])$, as the annihilator of the subspace $\Ker F'(x^*)$, coincides with the
image of the adjoint operator of $F'[x^*]$ (\cite{HutsonPym}, theorem 6.5.10), i.~e. for any 
$p \in N(0, \Ker F'[x^*])$ there exists $y^* \in Y^*$ such that $p = y^* \circ F'[x^*]$.
\end{proof}
\end{example}

\begin{example}
Let $H$ consist of all proper l.s.c. convex functions $h \colon X \to \overline{\mathbb{R}}$ such that 
$0 \in \interior \dom h$. Let us recall that in this case if there exists an upper coexhauster of the function $f$ at
the point $x \in \Omega$, then the function $f$ is weakly $H$-hyperdifferentiable at this point (cf.~example
\ref{ExampleCoexhauster}).

Arguing in a similar way to the proof of proposition \ref{PrGenNCMCodiff} one can get the following result.

\begin{proposition} \label{PrGenNCMCoexhauster}
Let $Y$ be a Banach space, a mapping $F \colon X \to Y$ be continuously Fr\'echet differentiable at a point $x^* \in X$.
Suppose that there exist Fr\'echet upper coexhausters $\overline{E}_i (x^*)$ of the functions $f_i$ at a point $x^*$,
$i \in I_0$. Suppose also that the operator $F'[x^*]$ is surjective, and $x^*$ is a point of local minimum in 
the problem
\[
f_0(x) \to \inf, \quad F(x) = 0, \quad f_i(x) \le 0, \quad i \in I.
\]
Then for any $C_i \in \overline{e}_i(x^*)$, $i \in R(x^*)$ there exists $y^* \in Y^*$ such that
\[
(0, y^* \circ F'[x^*]) \in \Big( \co \bigcup_{i \in R(x^*)} C_i \Big).
\]
where $\overline{e}_i(x^*) = \{ C \in \overline{E}_i (x^*) \mid \max_{(a, p) \in C} a = 0 \}$.
\end{proposition}

Let us obtain necessary conditions for a maximum in terms of a lower coexhauster.

\begin{proposition}
Suppose that there exists an upper coexhauster $\overline{E}(x^*)$ of the function $f$ at a point $x^* \in A$, and let
$x^*$ be a point of local maximum of the function $f$ on the set $A$. Then for any $\varepsilon > 0$ and 
$g \in \gamma(x^*, A)$ there exists $C \in \overline{E}(x^*)$ such that $p(g) \le \varepsilon$ for all $(0, p) \in C$.
Moreover, if the family $\overline{E}(x^*)$ is finite, then for any $g \in \Gamma(x^*, A)$ there exists 
$C \in \overline{E}(x^*)$ such that $p(g) \le 0$ for all $(0, p) \in C$.
\end{proposition}

\begin{proof}
By virtue of theorem \ref{ThNessMaxCondExFamUCA} one has that for any $\varepsilon > 0$ and $g \in \gamma(x^*, A)$ there
exists $C \in \overline{E}(x^*)$ such that $h'(0, g) \le \varepsilon$, where 
$h(\cdot) = \max_{(a, p) \in C}(a + p(\cdot))$. It remain to note that 
$h'(0, g) = \max_{p \in \underline{\partial}h(0)} p(g)$, where 
$\underline{\partial}h(0) = \{ p \in X^* \mid [0, p] \in C \}$ (see, e.g., \cite{IoffeTihomirov}, chapter 4).
\end{proof}
\end{example}

\begin{remark}
One could also consider the case when the set $H$ consists of all proper u.s.c. concave functions 
$h \colon X \to \overline{\mathbb{R}}$ such that $0 \in \interior \dom h$.
\end{remark}

\section*{Acknowledgements}

The author is grateful to professor V.F. Demyanov for his support and help with getting acquainted with the
ideas of abstract convex analysis.

\bibliographystyle{abbrv}  
\bibliography{AbstrConvexApprox_bibl}

\end{document}